\documentclass[final,a4paper]{siamltex}
\usepackage{amsmath,amssymb,amsfonts}
\usepackage{color}
\usepackage{enumerate}
\usepackage{verbatim}
\usepackage{eepic,epic,epsfig,epstopdf}
\usepackage[]{graphics}
\usepackage{graphicx,inputenc}
\usepackage{subfigure}
\usepackage{changepage}
\usepackage{amsmath}
\usepackage{threeparttable}
\usepackage{multirow}
\usepackage{extarrows}
\usepackage{url,color}
\usepackage{graphicx,subfig}
\graphicspath{{./pics/}}
\usepackage{bm,color,url}
\usepackage{algorithm,algorithmic}
\usepackage{geometry}
\usepackage{setspace}
\usepackage{booktabs}
\usepackage[notref,notcite]{showkeys}

\newtheorem{Theorem}{Theorem}[section]
\newtheorem{Lemma}[Theorem]{Lemma}

\newtheorem{Definition}[Theorem]{Definition}
\newtheorem{Corollary}[Theorem]{Corollary}
\newtheorem{Remark}{Remark}[section]

\newcommand{\tabincell}[2]{\begin{tabular}{@{}#1@{}}#2\end{tabular}}

\allowdisplaybreaks[4]

\title{Error Analysis of Deep Ritz Methods for Elliptic  Equations}
\author{
	Yuling Jiao\thanks{School of Mathematics and Statistics, and
		Hubei Key Laboratory of Computational Science, Wuhan University, Wuhan 430072, P.R. China. (yulingjiaomath@whu.edu.cn)}\quad\and
	Yanming Lai \thanks{School of Mathematics and Statistics,  Wuhan University, Wuhan 430072, P.R. China. (laiyanming@whu.edu.cn)}\quad\and
	Yisu Lo\thanks{Department of Mathematics, The Hong Kong University of Science and Technology,
Clear Water Bay, Kowloon, Hong Kong (yloab@connect.ust.hk)}\quad\and
 Yang Wang \thanks{Department of Mathematics, The Hong Kong University of Science and Technology,
Clear Water Bay, Kowloon, Hong Kong (yangwang@ust.hk)}\quad \and
Yunfei Yang  \thanks{Department of Mathematics, The Hong Kong University of Science and Technology,
Clear Water Bay, Kowloon, Hong Kong (yyangdc@connect.ust.hk)}
}

\begin{document}
	\maketitle
	
	\begin{abstract}
Using deep neural networks to solve PDEs has attracted  a lot of attentions recently.
However, why the deep learning method works is falling  far behind its empirical success.
In this paper, we provide a rigorous numerical analysis on deep Ritz method  (DRM) \cite{Weinan2017The} for second order elliptic equations with Drichilet, Neumann and Robin boundary condition, respectively.
We establish the first   nonasymptotic convergence rate in $H^1$ norm   for   DRM  using   deep networks with smooth  activation functions including logistic and hyperbolic tangent functions.
 Our results show   how to set the  hyper-parameter  of depth and width to achieve the desired convergence rate in terms of number of training samples.
	\end{abstract}

% REQUIRED

% REQUIRED

\section{Introduction}

Partial differential equations (PDEs) are one of the fundamental mathematical models in studying a variety of phenomenons arising in science and engineering. There have been established many conventional numerical methods successfully for solving PDEs in the case of low dimension $(d\leq3)$, particularly the finite element method \cite{brenner2007mathematical,ciarlet2002finite,Quarteroni2008Numerical,Thomas2013Numerical,Hughes2012the}. However,
one will encounter some  difficulties in both of theoretical analysis and numerical implementation when extending conventional numerical schemes  to high-dimensional PDEs. %On the one hand, more complicated interior structure will be involved in the designing of finite element schemes as the dimensionality increases, e.g., studying the approximation of $H^m\ (m\geq2)$ problems \cite{ciarlet2002finite,brenner2007mathematical}.
The classic analysis of convergence, stability and any other properties will be trapped into troublesome situation due to the complex construction of finite element space \cite{ciarlet2002finite,brenner2007mathematical}. Moreover, in the term of practical computation, the scale of the discrete problem will increase exponentially with respect to the  dimension.
%Thus there are rarely applications and discussions for high-dimensional problems using conventional numerical methods.

 Motivated by the well-known fact that deep learning method for  high-dimensional data analysis has been achieved great successful applications in discriminative, generative and reinforcement learning \cite{he2015delving,Goodfellow2014Generative,silver2016mastering}, solving high dimensional PDEs with  deep neural networks becomes an extremely potential  approach and  has  attracted much attentions  \cite{Cosmin2019Artificial,Justin2018DGM,DeepXDE,raissi2019physics,Weinan2017The,Yaohua2020weak,Berner2020Numerically,Han2018solving}. Roughly speaking, these works can be divided into three categories. The first category is using deep neural network to improve  classical numerical methods, see for example  \cite{Kiwon2020Solver,Yufei2020Learning,hsieh2018learning,Greenfeld2019Learning}.
  %\cite{Kiwon2020Solver} focuses on iterative PDE solving algorithms, and use network to reduce numerical errors of them. \cite{hsieh2018learning} proposes a neural-network based approach to learn a fast iterative solver by modifying the updates of an existing solver using a deep neural network. In \cite{Greenfeld2019Learning}, a framework for learning
%multigrid solvers is proposed, which learns a mapping from a family of parameterized PDEs to prolongation operators.
In the second category, the neural operator is introduced to learn mappings between infinite-dimensional spaces with neural networks \cite{Li2020Advances,anandkumar2020neural,li2021fourier}.
%In \cite{anandkumar2020neural}, the neural operator, instantiated through graph kernel networks, is used to map input data to PDEs and their solutions.
 For the last category, one utilizes deep neural networks to approximate the solutions of PDEs directly including physics-informed neural networks (PINNs) \cite{raissi2019physics}, deep Ritz method (DRM)  \cite{Weinan2017The} and deep  Galerkin method (DGM) \cite{Yaohua2020weak}. PINNs  is based on residual minimization for solving PDEs  \cite{Cosmin2019Artificial,Justin2018DGM,DeepXDE,raissi2019physics}.
 %Various model problems in literature have been studied such as \cite{raissi2019physics}, the paper is devoted to minimize the cost functional which consists of the squared residuals on the domain and boundary penalty, using so-called physics-informed neural networks (PINNs).
 %Related extensions of PINNs are referred to \cite{jagtap2020conservative,lagaris1998artificial,fpinns,npinns}.
 Proceed from the variational form, \cite{Weinan2017The,Yaohua2020weak,Xu2020finite} propose neural-network based methods related to classical Ritz and Galerkin method. In \cite{Yaohua2020weak}, weak adversarial networks (WAN) are proposed inspired by Galerkin method. Based on Ritz method, \cite{Weinan2017The} proposes the deep Ritz method (DRM) to solve variational problems corresponding to a class of PDEs.

\subsection{Related works and contributions}
\par
The idea of using neural networks to solve PDEs goes back to 1990's \cite{Isaac1998Artificial,Dissanayake1994neural}.
Although there are great empirical achievements in recent
several years, a challenging and interesting  question is to provide a
rigorous error analysis such as finite element method.
Several recent efforts  have been devoted to making processes along this line,  see for example  \cite{e2020observations,Luo2020TwoLayerNN,Mishra2020EstimatesOT,Mller2021ErrorEF,lu2021priori,hong2021rademacher,Shin2020ErrorEO,Wang2020WhenAW,e2021barron}.
%\cite{shin2020convergence,mishra2020estimates,shin2020error,Wang2020WhenAW} studies the convergence of  PINNs.
%In the past years, a great deal of effort has been devoted to the study of approximation error and generalization error associated with Barron class \cite{e2020observations,e2021barron}. There have been extensive applications in analysis of neural-network-based methods, e.g. \cite{Luo2020TwoLayerNN,lu2021priori,hong2021rademacher}.
In \cite{Luo2020TwoLayerNN}, least squares minimization method with two-layer neural networks is studied, the optimization error under the assumption of over-parametrization and generalization error without the over-parametrization assumption are analyzed.
In \cite{lu2021priori,Xu2020finite,hong2021priori}, the generalization error bounds of two-layer neural networks are derived
%and the convergence rates of generalization errors are proved independent of the dimension $d$,
via assuming that the exact solutions lie in spectral Barron space.
Although the  studies  in \cite{lu2021priori,Xu2020finite} can overcome the curse of dimensionality,
it should be pointed out that it is difficult to generalize these results to deep neural networks or to the situation where the underlying solutions are living  in general Sobolev spaces.
 %and, more importantly, we know nothing about the situation where do not hold.
%Limited work has been done in analysis of convergence rate without assumptions about Barron spaces.

\textbf{ Two  important   questions have not been addressed   in the above mentioned related study are those:  Can we provide a convergence result of DRM only requiring the target solution living in $H^2$ ?  How to determine the depth and width to achieve the  desired convergence rate ?}
In this paper, we give a firm answers on these  questions by  providing  an error analysis of using DRM with  sigmoidal deep neural networks  to solve second order elliptic  equations with  Drichilet, Neumann and Robin boundary condition, respectively.

Let $u_{\phi_\mathcal{A}}$ be the solution of a random solver for DRM (i.e., $u_{\phi_\mathcal{A}}$ is the solution of $(\ref{optimization})$) and use the notation $\mathcal{N}_{\rho}\left(\mathcal{D}, \mathfrak{n}_{\mathcal{D}}, B_{\theta}\right)$ to refer to the collection of functions implemented by a $\rho-$neural network with depth $\mathcal{D}$, total number of nonzero weights $\mathfrak{n}_{\mathcal{D}}$ and each weight being bounded by $B_{\theta}$. Set $\rho=\frac{1}{1+e^{-x}}$ or  $\frac{e^x-e^{-x}}{e^x+e^{-x}}$. Our main contributions  are as follows:
\begin{itemize}
\item
Let $u_R$ be the weak solution of Robin problem $(\ref{second order elliptic equation})(\ref{robin})$. For any $\epsilon>0$ and $\mu \in (0,1)$, set
the parameterized function class
\begin{equation*}
\mathcal{P}=\mathcal{N}_{\rho}\left(C\log(d+1),C(d,\beta)\epsilon^{-d /(1-\mu )},C(d,\beta) \epsilon^{-(9d+8)/(2-2\mu)}\right)
\end{equation*}
and number of samples
\begin{equation*}
N=M=C(d,\Omega, coe,\alpha,\beta)\epsilon^{-Cd\log(d+1)/(1-\mu)},
\end{equation*}
if the optimization error of $u_{\phi_\mathcal{A}}$ is $\mathcal{E}_{opt}\le \epsilon$, then
\begin{equation*}
\mathbb{E}_{\{{X_i}\}_{i=1}^{N},\{{Y_j}\}_{j=1}^{M}}\|u_{\phi_\mathcal{A}}-u_R\|_{H^1(\Omega)}\leq C(\Omega, coe,\alpha)\epsilon.
\end{equation*}
\item Let $u_D$ be the weak solution of Dirichlet problem $(\ref{second order elliptic equation})(\ref{dirichlet})$. Set $\alpha=1,g=0$. For any $\epsilon>0$, let $\beta=C(coe)\epsilon$ as the penalty parameter, set the parameterized function class
\begin{equation*}
\mathcal{P}=\mathcal{N}_{\rho}\left(C\log(d+1),C(d)\epsilon^{-5d /2(1-\mu )},C(d) \epsilon^{-(45d+40)/(4-4\mu)}\right)
\end{equation*}
and number of samples
\begin{equation*}
N=M=C(d,\Omega, coe)\epsilon^{-Cd\log(d+1)/(1-\mu)},
\end{equation*}
if the optimization error $\mathcal{E}_{opt} \le \epsilon$, then
\begin{equation*}
\mathbb{E}_{\{{X_i}\}_{i=1}^{N},\{{Y_j}\}_{j=1}^{M}}\|u_{\phi_\mathcal{A}}-u_D\|_{H^1(\Omega)}\leq C(\Omega, coe)\epsilon.
\end{equation*}
\end{itemize}

We summarize the related works and our results in the following table \ref{RWK}.

{\small
\begin{table}[ht!]
	
	\centering
	\begin{threeparttable}
		\caption{Previous works and our result}\label{RWK}
		\label{compare}
		\begin{tabular}{ccccc}
			\toprule
			Paper&\tabincell{c}{Depth and\\ activation functions }&Equation(s)&\tabincell{c}{Regularity  \\ Condition}\cr
			\midrule
			\cite{Luo2020TwoLayerNN}&\tabincell{c}{$\mathcal{D}=2$\\ $\mathrm{ReLU}^3$}&\tabincell{c}{Second order \\ differential equation}&$u^*\in \mathrm{Barron}\ \  \mathrm{class}$\cr
\midrule
			\cite{lu2021priori}&\tabincell{c}{$\mathcal{D}=2$\\ Softplus}&\tabincell{c}{Poisson equation \\ and Schr$\mathrm{\ddot{o}}$dinger equation}&$u^*\in \mathrm{Barron} \ \  \mathrm{class}$\cr
\midrule
			\cite{hong2021priori}&\tabincell{c}{$\mathcal{D}=2$\\$\mathrm{ReLU}^k$}&\tabincell{c}{2$m$-th order \\ differential equation}&$u^*\in \mathrm{Barron}\ \  \mathrm{class}$\cr
\midrule
%			\cite{Shin2020ErrorEO}&\tabincell{c}{$\mathcal{D}=2$\\Activations satisfing \\ assumption 2.7, \\ e.g., $\mathrm{ReLU}$}&\tabincell{c}{Elliptic, advection-reaction, \\
%	        integro-differential \\ equations}&No&No\cr
%\midrule
            \cite{duan2021convergence}&\tabincell{c}{$\mathcal{D}=\mathcal{O}(\log d)$ \\ $\mathrm{ReLU}^2$}&\tabincell{c}{Second order \\ elliptic equation}&\tabincell{c}{$u^*\in C^2$}\cr
            \midrule
            This paper&\tabincell{c}{$\mathcal{D}=\mathcal{O}(\log d)$ \\ Logistic and \\ Hyperbolic tangent}&\tabincell{c}{Second order \\ elliptic equation}&$u^*\in H^2$\cr
			\bottomrule
		\end{tabular}
	\end{threeparttable}
\end{table}
}

The rest of the paper are organized as follows.  In Section 2,  we give some preliminaries.  In Section 3, we  present the DRM method and the error decomposition results for analysis of DRM. In Section 4 and 5, we give   detail analysis
on the approximation error and statistical error. In Section 6, we present our main results.
We give conclusion and short discussion in Section 7.

\section{Neural Network}
Due to its strong expressivity, neural network function class plays an important role in machine learning. A variety of neural networks are choosen as parameter function classes in the training process. We now introduce some notation related to neural network which will simplify our later discussion. Let $\mathcal{D}\in\mathbb{N}^+$. A function $\mathbf{f}: \mathbb{R}^{d} \rightarrow \mathbb{R}^{n_{\mathcal{D}}}$ implemented by a neural network is defined by
\begin{equation}\label{nn}
	\begin{array}{l}
		\mathbf{f}_{0}(\mathbf{x})=\mathbf{x},\\
		\mathbf{f}_{\ell}(\mathbf{x})=\mathbf{\rho}\left(A_{\ell} \mathbf{f}_{\ell-1}+\mathbf{b}_{\ell}\right)
		\quad \text { for } \ell=1, \ldots, \mathcal{D}-1, \\
		\mathbf{f}:=\mathbf{f}_{\mathcal{D}}(\mathbf{x})=A_{\mathcal{D}}\mathbf{f}_{\mathcal{D}-1}+\mathbf{b}_{\mathcal{D}},
	\end{array}
\end{equation}
where $A_{\ell}=\left(a_{ij}^{(\ell)}\right)\in\mathbb{R}^{n_{\ell}\times n_{\ell-1}}$ and $\mathbf{b}_{\ell}=\left(b_i^{(\ell)}\right)\in\mathbb{R}^{n_{\ell}}$. $\rho$ is called the activation function and acts componentwise. $\mathcal{D}$ is called the depth of the network and $\mathcal{W}:=\max\{n_{\ell}:\ell=1,\cdots,\mathcal{D}\}$ is called the width of the network. $\phi = \{A_{\ell},\mathbf{b}_{\ell}\}_{\ell}$ are called the weight parameters. For convenience, we denote $\mathfrak{n}_i$, $i=1,\cdots,\mathcal{D}$, as the number of nonzero weights on the first $i$ layers in the representation (\ref{nn}). Clearly $\mathfrak{n}_{\mathcal{D}}$ is the total number of nonzero weights. Sometimes we denote a function implemented by a neural network as $\mathbf{f}_{\rho}$ for short. We use the notation $\mathcal{N}_{\rho}\left(\mathcal{D}, \mathfrak{n}_{\mathcal{D}}, B_{\theta}\right)$ to refer to the collection of functions implemented by a $\rho-$neural network with depth $\mathcal{D}$, total number of nonzero weights $\mathfrak{n}_{\mathcal{D}}$ and each weight being bounded by $B_{\theta}$.

\section{Deep Ritz Method and Error Decomposition}
	Let $\Omega$ be a convex bounded open set in $\mathbb{R}^d$ and assume that $\partial\Omega\in C^{\infty}$. Without loss of generality we assume that $\Omega\subset[0,1]^d$. We consider the following second order  elliptic equation:
\begin{equation} \label{second order elliptic equation}
		-\triangle u +  w u = f  \text { in } \Omega
\end{equation}
with three kinds of boundary condition:
\begin{subequations}
\begin{align}
	u&=0\text { on } \partial \Omega\label{dirichlet}\\
	\frac{\partial u}{\partial n}&=g\text { on } \partial \Omega\label{neumann}\\
	\alpha u+\beta\frac{\partial u}{\partial n}&=g\text { on } \partial \Omega,\quad\alpha,\beta\in\mathbb{R},\beta\neq0\label{robin}
\end{align}
\end{subequations}
which are called Drichilet, Neumann and Robin boundary condition, respectively. Note that for Drichilet problem, we only consider the homogeneous boundary condition here since the inhomogeneous case can be turned into homogeneous case by translation. We also remark that Neumann condition $(\ref{neumann})$ is covered by Robin condition $(\ref{robin})$. Hence in the following we only consider Dirichlet problem and Robin problem.

We make the following assumption on the known terms in equation:
\begin{enumerate}[(A)]
\item $\qquad f\in L^\infty(\Omega)$, $g\in H^{1/2}(\Omega)$, $w\in L^{\infty}(\Omega)$, $w\geq c_w$
\end{enumerate}
where $c_w$ is some positive constant. In the following we abbreviate $$C\left(\|f\|_{L^\infty(\Omega)},\|g\|_{H^{1/2}(\Omega)},\|w\|_{L^{\infty}(\Omega)},c_w\right),$$  constants depending on the known terms in equation, as $C(coe)$ for simplicity.

For problem $(\ref{second order elliptic equation})(\ref{dirichlet})$, the varitional problems is to find $u\in H_0^1(\Omega)$ such that
\begin{subequations}
\begin{equation} \label{variational dirichlet}
(\nabla u,\nabla v)+(wu,v)=(f,v),\quad\forall v\in H_{0}^1(\Omega).
\end{equation}
The corresponding minimization problem is
\begin{equation} \label{minimization dirichlet}
\min_{u\in H_0^1(\Omega)}\frac{1}{2}\int_{\Omega}\left(|\nabla u|^2+wu^2-2fu\right)dx.
\end{equation}
\end{subequations}
\begin{Lemma} \label{uD regularity}
Let (A) holds. Let $u_D$ be the solution of problem $(\ref{variational dirichlet})$(also $(\ref{minimization dirichlet})$). Then $u_D\in H^2(\Omega)$.
\end{Lemma}
\begin{proof}
See \cite{evans1998partial}.
\end{proof}

For problem $(\ref{second order elliptic equation})(\ref{robin})$, the variational problem is to find $u\in H^1(\Omega)$ such that
\begin{subequations}
\begin{equation} \label{variational robin}
(\nabla u,\nabla v)+(wu,v)+\frac{\alpha}{\beta}(T_0u,T_0v)|_{\partial\Omega}=(f,v)+\frac{1}{\beta}(g,T_0v)|_{\partial\Omega},\quad\forall v\in H^1(\Omega)
\end{equation}
where $T_0$ is zero order trace operator. The corresponding minimization problem is
\begin{equation} \label{minimization robin}
\min_{u\in H^1(\Omega)}\int_{\Omega}\left(\frac{1}{2}|\nabla u|^2+\frac{1}{2}wu^2-fu\right)dx
+\frac{1}{\beta}\int_{\partial\Omega}\left(\frac{\alpha}{2}(T_0u)^2-gT_0u\right)ds.
\end{equation}
\end{subequations}
\begin{Lemma} \label{uR regularity}
	Let (A) holds. Let $u_R$ be the solution of problem $(\ref{variational robin})$(also $(\ref{minimization robin})$). Then $u_R\in H^2(\Omega)$ and $\|u_R\|_{H^2(\Omega)}\leq\frac{ C(coe)}{\beta}$ for any $\beta>0$.
\end{Lemma}
\begin{proof}
See  \cite{evans1998partial}.
\end{proof}

Intuitively, when $\alpha=1,g=0$ and $\beta\to0$, we expect that the solution of Robin problem converges to the solution of Dirichlet problem. Hence we only need to consider the Robin problem since the Dirichlet problem can be handled through a limit process. Define $\mathcal{L}$ as a functional on $H^1(\Omega)$:
\begin{equation*}
	\mathcal{L}(u):=\int_{\Omega}\left(\frac{1}{2}|\nabla u|^2+\frac{1}{2}wu^2-fu\right)dx
	+\frac{1}{\beta}\int_{\partial\Omega}\left(\frac{\alpha}{2}(T_0u)^2-gT_0u\right)ds.
\end{equation*}
The next lemma verify the assertion.
\begin{Lemma} \label{penalty convergence}
Let (A) holds. Let $\alpha=1,g=0$. Let $u_D$ be the solution of problem $(\ref{variational dirichlet})$(also $(\ref{minimization dirichlet})$) and $u_R$ the solution of problem $(\ref{variational robin})$(also $(\ref{minimization robin})$). There holds
\begin{equation*}
\|u_R-u_D\|_{H^1(\Omega)}\leq C(coe)\beta.
\end{equation*}
\end{Lemma}
\begin{proof}
We first have
\begin{equation*}
\int_{\Omega}\nabla u_D\cdot\nabla vdx-\int_{\partial\Omega}T_1u_Dvds+\int_{\Omega}wu_Dvdx=\int_{\Omega}fvdx,\quad\forall v\in H^1(\Omega).
\end{equation*}
with $T_1$ being first order trace operator. Hence for any $u\in H^1(\Omega)$,
\begin{align}
	\mathcal{L}(u)&=\int_{\Omega}\left(\frac{1}{2}|\nabla u|^2+\frac{1}{2}wu^2-fu\right)dx
	+\frac{1}{2\beta}\int_{\partial\Omega}(T_0u)^2ds\nonumber\\
	&=\int_{\Omega}\left(\frac{1}{2}|\nabla u|^2+\frac{1}{2}wu^2\right)dx
	+\frac{1}{2\beta}\int_{\partial\Omega}(T_0u)^2ds-\int_{\Omega}\nabla u_D\cdot\nabla udx+\int_{\partial\Omega}T_1u_Duds-\int_{\Omega}wu_Dudx,\nonumber\\
	&=\int_{\Omega}\left(\frac{1}{2}|\nabla u-\nabla u_D|^2+\frac{1}{2}w(u-u_D)^2\right)dx
	+\frac{1}{2\beta}\int_{\partial\Omega}\left(T_0u+\beta T_1u_D\right)^2ds\nonumber\\
	&\quad-\int_{\Omega}\left(\frac{1}{2}|\nabla u_D|^2+\frac{1}{2}wu_D^2\right)dx-\frac{\beta}{2}\int_{\partial\Omega}\left(T_1u_D\right)^2ds.\label{penalty convergence1}
\end{align}
Define
\begin{equation*}
\mathcal{R}_{\beta}(u)=\int_{\Omega}\left(\frac{1}{2}|\nabla u-\nabla u_D|^2+\frac{1}{2}w(u-u_D)^2\right)dx
+\frac{1}{2\beta}\int_{\partial\Omega}\left(T_0u+\beta T_1u_D\right)^2ds.
\end{equation*}
Since $u_R$ is the minimizer of $\mathcal{L}$, from $(\ref{penalty convergence1})$ we conclude that it is also the minimizer of $\mathcal{R}$.

Note $u_D\in H^2(\Omega)$(Lemma $\ref{uD regularity}$), by trace theorem we know $T_1u_D\in H^{1/2}(\partial\Omega)$ and hence there exists $\phi\in H^1(\Omega)$ such that $T_0\phi=-T_1u_D$. Set $\bar{u}=\beta\phi+u_D$, then
\begin{equation*}
C(coe)\|u_R-u_D\|_{H^1(\Omega)}^2\leq\mathcal{R}(u_R)\leq\mathcal{R}(\bar{u})=\beta^2\int_{\Omega}\left(\frac{1}{2}|\nabla \phi|^2+\frac{1}{2}w\phi^2\right)dx=C(u_D,coe)\beta^2.
\end{equation*}
\end{proof}

Note that $\mathcal{L}$ can be equivalently written as
\begin{align*}
	\mathcal{L}(u)=&|\Omega|\mathbb{E}_{X\sim U(\Omega)}\left(\frac{1}{2}|\nabla u(X)|^2+\frac{1}{2}w(X)u^2(X)-f(X)u(X)\right)\\
	&+\frac{|\partial\Omega|}{\beta}\mathbb{E}_{Y\sim U(\partial\Omega)}\left(\frac{\alpha}{2}(T_0u)^2(Y)-g(Y)T_0u(Y)\right)
\end{align*}
where $U(\Omega)$ and $U(\partial\Omega)$ are uniform distribution on $\Omega$ and $\partial\Omega$, respectively. We then introduce a discrete version of $\mathcal{L}$:
\begin{align*}
	\widehat{\mathcal{L}}(u):=&\frac{|\Omega|}{N}\sum_{i=1}^{N}\left(\frac{1}{2}|\nabla u(X_i)|^2+\frac{1}{2}w(X_i)u^2(X_i)-f(X_i)u(X_i)\right)\\
	&+\frac{|\partial\Omega|}{\beta M}\sum_{j=1}^{M}\left(\frac{\alpha}{2}(T_0u)^2(Y_j)-g(Y_j)T_0u(Y_j)\right)
\end{align*}
where $\{X_i\}_{i=1}^{N}$ and $\{Y_j\}_{j=1}^{M}$ are i.i.d. random variables according to $U(\Omega)$ and $U(\partial\Omega)$ respectively. We now consider a minimization problem with respect to $\widehat{\mathcal{L}}$:
\begin{equation} \label{optimization}
\min_{u_{\phi}\in\mathcal{P}}\widehat{\mathcal{L}}(u_{\phi})
\end{equation}
where $\mathcal{P}$ refers to the parameterized function class. We denote by $\widehat{u}_{\phi}$ the solution of problem $(\ref{optimization})$. Finally, we call a (random) solver $\mathcal{A}$, say SGD, to
minimize $\widehat{\mathcal{L}}$ and denote the output of $\mathcal{A}$, say $u_{\phi_\mathcal{A}}$, as the final solution.

The following error decomposition enables us to apply different methods to deal with different kinds of error.
\begin{proposition} \label{error decomposition}
	Let (A) holds. Assume that $\mathcal{P}\subset H^1(\Omega)$. Let $u_R$ and $u_D$ be the solution of problem $(\ref{variational robin})$(also $(\ref{minimization robin})$) and $(\ref{variational dirichlet})$(also $(\ref{minimization dirichlet})$), respectively. Let $u_{\phi_{\mathcal{A}}}$ be the solution of problem $(\ref{optimization})$ generated by a random solver.

(1)
$$
\|u_{\phi_\mathcal{A}}-u_R\|_{H^1(\Omega)}\leq C(coe)\left[\mathcal{E}_{app} +\mathcal{E}_{sta} + \mathcal{E}_{opt}\right]^{1/2}.
$$
where
\begin{align*}
\mathcal{E}_{app} &= \frac{1}{\beta}C(\Omega,coe,\alpha)\inf_{\bar{u}\in\mathcal{P}}\|\bar{u}-u_R\|_{H^1(\Omega)}^2, \\
\mathcal{E}_{sta} &=  \sup _{u \in \mathcal{P}} \left[\mathcal{L}(u)-\widehat{\mathcal{L}}(u)\right] + \sup _{u \in \mathcal{P}} \left[\widehat{\mathcal{L}}(u) - \mathcal{L}(u)\right], \\
\mathcal{E}_{opt} &= \widehat{\mathcal{L}}\left(u_{\phi_{\mathcal{A}}}\right)-\widehat{\mathcal{L}}\left(\widehat{u}_{\phi}\right).
\end{align*}

(2) Set $\alpha=1,g=0$.
$$
\|u_{\phi_\mathcal{A}}-u_D\|_{H^1(\Omega)} \leq C(coe)\left[\mathcal{E}_{app} +\mathcal{E}_{sta} + \mathcal{E}_{opt}
+\|u_R-u_D\|_{H^1(\Omega)}^2  \right]^{1/2}.
$$
\end{proposition}
\begin{proof}
We only prove (1) since (2) is a direct result from (1) and triangle inequality. For any $u\in\mathcal{P}$, set $v=u-u_R$, then
\begin{align*}
	&\mathcal{L}\left(u\right)=\mathcal{L}\left(u_R+v\right)\\
	&= \frac{1}{2}(\nabla (u_R+v),\nabla (u_R+v))_{L^{2}(\Omega)}+\frac{1}{2}(u_R+v,u_R+v)_{L^{2}(\Omega;w)}-\langle u_R+v, f\rangle_{L^2({\Omega})}\\
	&\quad\ +\frac{\alpha}{2\beta}(T_0u_R+T_0v,T_0u_R+T_0v)_{L^2({\partial \Omega})}- \frac{1}{\beta}\langle {T_0u_R+T_0v}, g\rangle_{L^2({\partial \Omega})}\\
	&=\frac{1}{2}(\nabla u_R,\nabla u_R)_{L^{2}(\Omega)}+\frac{1}{2}(u_R,u_R)_{L^{2}(\Omega;w)}-\langle u_R, f\rangle_{L^2({\Omega})} +\frac{\alpha}{2\beta}(T_0u_R,T_0u_R)_{L^2({\partial \Omega})}\\
	&\quad - \frac{1}{\beta}\langle {T_0u_R}, g\rangle_{L^2({\partial \Omega})} +\frac{1}{2}(\nabla v,\nabla v)_{L^{2}(\Omega)}+\frac{1}{2}(v,v)_{L^{2}(\Omega;w)}+\frac{\alpha}{2\beta}(T_0v,T_0v)_{L^2({\partial \Omega})}\\
	&\quad\ +\left[(\nabla u_R,\nabla v)_{L^{2}(\Omega)}+(u^*,v)_{L^{2}(\Omega;w)}-\langle v, f\rangle_{L^2({\Omega})} +\frac{\alpha}{\beta}(T_0u_R,T_0v)_{L^2({\partial \Omega})}- \frac{1}{\beta}\langle {T_0v}, g\rangle_{L^2({\partial \Omega})}\right]\\
	&=\mathcal{L}\left(u_R\right)+\frac{1}{2}(\nabla v,\nabla v)_{L^{2}(\Omega)}+\frac{1}{2}(v,v)_{L^{2}(\Omega;w)}+\frac{\alpha}{2\beta}(T_0v,T_0v)_{L^2({\partial \Omega})},
\end{align*}
where the last equality is due to the fact that $u_R$ is the solution of equation $(\ref{variational robin})$. Hence
\begin{align*}
	C(coe)\|v\|_{H^1(\Omega)}^2&\leq\mathcal{L}\left(u\right)-\mathcal{L}\left(u_R\right)
	=\frac{1}{2}(\nabla v,\nabla v)_{L^{2}(\Omega)}+\frac{1}{2}(v,v)_{L^{2}(\Omega;w)}+\frac{\alpha}{2\beta}(T_0v,T_0v)_{L^2({\partial \Omega})}\\
	&\leq \frac{1}{\beta}C(\Omega,coe,\alpha)\|v\|_{H^1(\Omega)}^2,
\end{align*}
here we apply trace inequality
\begin{equation*}
\|T_0v\|_{L^2(\partial\Omega)}\leq C(\Omega)\|v\|_{H^1(\Omega)}.
\end{equation*}
See more details in \cite{adams2003sobolev}. In other words, we obtain
\begin{equation} \label{errdec1}
C(coe)\|u-u_R\|_{H^1(\Omega)}^2\leq\mathcal{L}\left(u\right)-\mathcal{L}\left(u_R\right)
\leq \frac{1}{\beta}C(\Omega,coe,\alpha)\|u-u_R\|_{H^1(\Omega)}^2.
\end{equation}

Now, letting $\bar{u}$ be any element in $\mathcal{P}$, we have
\begin{equation*}
	\begin{split}
		&\mathcal{L}\left(u_{\phi_{\mathcal{A}}}\right)-\mathcal{L}\left(u_R\right) \\
		=&\mathcal{L}\left(u_{\phi_{\mathcal{A}}}\right)-\widehat{\mathcal{L}}\left(u_{\phi_{\mathcal{A}}}\right) +\widehat{\mathcal{L}}\left(u_{\phi_{\mathcal{A}}}\right)-\widehat{\mathcal{L}}\left(\widehat{u}_{\phi}\right) +\widehat{\mathcal{L}}\left(\widehat{u}_{\phi}\right)-\widehat{\mathcal{L}}\left(\bar{u}\right) +\widehat{\mathcal{L}}\left(\bar{u}\right)-\mathcal{L}\left(\bar{u}\right)+\mathcal{L}\left(\bar{u}\right)-\mathcal{L}\left(u_R\right) \\
		\leq & \sup _{u \in \mathcal{P}} \left[\mathcal{L}(u)-\widehat{\mathcal{L}}(u)\right] +\left[\widehat{\mathcal{L}}\left(u_{\phi_{\mathcal{A}}}\right)-\widehat{\mathcal{L}}\left(\widehat{u}_{\phi}\right)\right] + \sup _{u \in \mathcal{P}} \left[\widehat{\mathcal{L}}(u) - \mathcal{L}(u)\right]  + \frac{1}{\beta}C(\Omega,coe,\alpha)\|\bar{u}-u_R\|_{H^1(\Omega)}^2,
	\end{split}
\end{equation*}
where the last step is due to inequality $(\ref{errdec1})$ and the fact that $\widehat{\mathcal{L}}\left(\widehat{u}_{\phi}\right)-\widehat{\mathcal{L}}\left(\bar{u}\right)\leq 0$. Since $\bar{u}$ can be any element in $\mathcal{P}$, we take the infimum of $\bar{u}$ on both side of the above display,
\begin{align}
	\mathcal{L}\left(u_{\phi_{\mathcal{A}}}\right)-\mathcal{L}\left(u_R\right)
	&\leq \inf_{\bar{u}\in\mathcal{P}}\frac{1}{\beta}C(\Omega,coe,\alpha)\|\bar{u}-u_R\|_{H^1(\Omega)}^2 + \sup _{u \in \mathcal{P}} [\mathcal{L}(u)-\widehat{\mathcal{L}}(u)] \nonumber\\
	& \quad + \sup _{u \in \mathcal{P}} [\widehat{\mathcal{L}}(u) - \mathcal{L}(u)]  +\left[\widehat{\mathcal{L}}\left(u_{\phi_{\mathcal{A}}}\right)-\widehat{\mathcal{L}}\left(\widehat{u}_{\phi}\right)\right].\label{errdec2}
\end{align}

Combining $(\ref{errdec1})$ and $(\ref{errdec2})$ yields the result.
\end{proof}

\section{Approximation Error}
For nerual network approximation in Sobolev spaces, \cite{guhring2021approximation} is a comprehensive study concerning a variety of activation functions, including $\mathrm{ReLU}$, sigmoidal type functions, etc. The key idea in \cite{guhring2021approximation} to build the upper bound in Sobolev spaces is to construct an approximate partition of unity.

Denote $\mathcal{F}_{s,p,d}:=\left\{f\in W^{s,p}\left([0,1]^d\right):\|f\|_{W^{s,p}\left([0,1]^d\right)}\leq1\right\}$.

\begin{Theorem}[Proposition 4.8, \cite{guhring2021approximation}]\label{general app}
Let $p\geq1$, $s,k,d\in\mathbb{N}^+$, $s\geq k+1$. Let $\rho$ be logistic function $\frac{1}{1+e^{-x}}$ or tanh function $\frac{e^x-e^{-x}}{e^x+e^{-x}}$. For any $\epsilon>0$ and $f\in\mathcal{F}_{s,p,d}$, there exists a neural network $f_{\rho}$ with depth $C\log (d+s)$ and $C(d,s,p,k)\epsilon^{-d /(s-k-\mu k)}$ non-zero weights such that
\begin{equation*}
\|f-f_{\rho}\|_{W^{k,p}\left([0,1]^d\right)}\leq\epsilon.
\end{equation*}
Moreover, the weights in the neural network are bounded in absolute value by
\begin{equation*}
C(d,s,p,k)\epsilon^{-2-\frac{2(d / p+d+k+\mu k)+d / p+d}{s-k-\mu k}}
\end{equation*}
where $\mu$ is an arbitrarily small positive number.
\end{Theorem}

\begin{Remark}
\textnormal{
The bounds in the theorem can be found in the proof of \cite[Proposition 4.8]{guhring2021approximation}, except that they did not explicitly give the bound on the depth. In their proof, they partition $[0,1]^d$ into small patches, approximate $f$ by a sum of localized polynomial $\sum_m \phi_m p_m$, and approximately implement $\sum_m \phi_m p_m$ by a neural network, where the bump functions $\{\phi_m\}$ form an approximately partition of unity and $p_m = \sum_{|\alpha|<s}c_{f,m,\alpha}x^\alpha$ are the averaged Taylor polynomials. As shown in \cite{guhring2021approximation}, $\phi_m$ can be approximated by the products of the $d$-dimensional output of a neural network with constant layers. And the identity map $I(x)=x$ and the product function $\times(a,b) = ab$ can also be approximated by neural networks with constant layers. In order to approximate $\phi_m x^\alpha$, we need to implement $d+s-1$ products. Hence, the required depth can be bounded by $C\log(d+s)$.
}
\end{Remark}

Since the region $[0,1]^d$ is larger than the region $\Omega$ we consider(recall we assume without loss of generality that $\Omega\subset[0,1]^d$ at the beginning), we need the following extension result.
\begin{Lemma} \label{extension}
Let $k\in\mathbb{N}^+$, $1\leq p<\infty$. There exists a linear operator $E$ from $W^{k,p}(\Omega)$ to $W_0^{k,p}\left([0,1]^d\right)$ and $Eu=u$ in $\Omega$.
\end{Lemma}
\begin{proof}
See Theorem 7.25 in \cite{gilbarg2015elliptic}.
\end{proof}

From Lemma $\ref{uR regularity}$ we know that our target function $u_R\in H^2(\Omega)$. Hence we are able to obtain an approximation result in $H^1$-norm.
\begin{Corollary} \label{app error}
Let $\rho$ be logistic function $\frac{1}{1+e^{-x}}$ or tanh function $\frac{e^x-e^{-x}}{e^x+e^{-x}}$. For any $\epsilon>0$ and $f\in H^2(\Omega)$ with $\|f\|_{H^2(\Omega)}\leq1$, there exists a neural network $f_{\rho}$ with depth $C\log(d+1)$ and $C(d)\epsilon^{-d /(1-\mu )}$ non-zero weights such that
\begin{equation*}
	\|f-f_{\rho}\|_{H^1\left(\Omega\right)}\leq\epsilon.
\end{equation*}
Moreover, the weights in the neural network are bounded by $C(d) \epsilon^{-(9d+8)/(2-2\mu)}$, where $\mu$ is an arbitrarily small positive number.
\end{Corollary}
\begin{proof}
Set $k=1$, $s=2$, $p=2$ in Theorem $\ref{general app}$ and use the fact $\|f-f_\rho\|_{H^2(\Omega)}\leq \|Ef-f_\rho\|_{H^2\left([0,1]^d\right)}$, where $E$ is the  extension operator in Lemma $\ref{extension}$.
\end{proof}

\section{Statistical Error}
In this section we investigate statistical error $\sup _{u \in \mathcal{P}} \pm\left[\mathcal{L}(u)-\widehat{\mathcal{L}}(u)\right]$.
\begin{Lemma} \label{triangle inequality}
\begin{equation*}
	\mathbb{E}_{\{{X_i}\}_{i=1}^{N},\{{Y_j}\}_{j=1}^{M}}\sup_{u\in\mathcal{P}}\pm\left[\mathcal{L}(u)-\widehat{\mathcal{L}}(u)\right]
	\leq\sum_{k=1}^{5}\mathbb{E}_{\{{X_i}\}_{i=1}^{N},\{{Y_j}\}_{j=1}^{M}}\sup_{u\in\mathcal{P}}\pm\left[\mathcal{L}_k(u)-\widehat{\mathcal{L}}_k(u)\right].
\end{equation*}
where
\begin{align*}
	&\mathcal{L}_1(u)=\frac{|\Omega|}{2}\mathbb{E}_{X\sim U(\Omega)}|\nabla u(X)|^2,
	&\mathcal{L}_2(u)&=\frac{|\Omega|}{2}\mathbb{E}_{X\sim U(\Omega)}w(X)u^2(X),\\
	&\mathcal{L}_3(u)=-|\Omega|\mathbb{E}_{X\sim U(\Omega)}f(X)u(X),
	&\mathcal{L}_4(u)&=\frac{\alpha|\partial\Omega|}{2\beta}\mathbb{E}_{Y\sim U(\partial\Omega)}(Tu)^2(Y),\\
	&\mathcal{L}_5(u)=-\frac{|\partial\Omega|}{\beta}\mathbb{E}_{Y\sim U(\partial\Omega)}g(Y)Tu(Y).
\end{align*}
	and $\widehat{\mathcal{L}}_k(u)$ is the discrete version of $\mathcal{L}_k(u)$, for example,
	\begin{equation*}
	\widehat{\mathcal{L}}_1(u)=\frac{|\Omega|}{2N}\sum_{i=1}^{N}|\nabla u(X_i)|^2.
	\end{equation*}
\end{Lemma}
\begin{proof}
Direct result from triangle inequality.
\end{proof}

By the technique of symmetrization, we can bound the difference between continuous loss $\mathcal{L}_i$ and empirical loss $\widehat{\mathcal{L}}_i$ by Rademacher complexity.
\begin{definition}
	The Rademacher complexity of a set $A \subseteq \mathbb{R}^N$ is defined as
	\begin{equation*}
		\mathfrak{R}_N(A) = \mathbb{E}_{\{\sigma_i\}_{k=1}^N}\left[\sup_{a\in A}\frac{1}{N}\sum_{k=1}^N \sigma_k a_k\right],
	\end{equation*}
	where,   $\{\sigma_k\}_{k=1}^N$ are $N$ i.i.d  Rademacher variables with $\mathbb{P}(\sigma_k = 1) = \mathbb{P}(\sigma_k = -1) = \frac{1}{2}.$
	The Rademacher complexity of  function class $\mathcal{F}$ associate with random sample $\{X_k\}_{k=1}^{N}$ is defined as
	\begin{equation*}
		\mathfrak{R}_N(\mathcal{F}) = \mathbb{E}_{\{X_k,\sigma_k\}_{k=1}^{N}}\left[\sup_{u\in \mathcal{F}}\frac{1}{N}\sum_{k=1}^N \sigma_k u(X_k)\right].
	\end{equation*}
\end{definition}

For Rademacher complexity, we have following structural result.
\begin{Lemma} \label{structural result of Rademacher}
	Assume that $w:\Omega\to\mathbb{R}$ and $|w(x)|\leq\mathcal{B}$ for all $x\in\Omega$, then for any function class $\mathcal{F}$, there holds
	\begin{equation*}
		\mathfrak{R}_N(w\cdot\mathcal{F})\leq\mathcal{B}\mathfrak{R}_N(\mathcal{F}),
	\end{equation*}
	where $w\cdot\mathcal{F}:=\{\bar{u}:\bar{u}(x)=w(x)u(x),u\in\mathcal{F}\}$.
\end{Lemma}
\begin{proof}
	\begin{align*}
		&\mathfrak{R}_N(w\cdot\mathcal{F})
		=\frac{1}{N}\mathbb{E}_{\{X_k,\sigma_k\}_{k=1}^{N}}\sup_{u\in\mathcal{F}}\sum_{k=1}^{N}\sigma_kw(X_k)u(X_k)\\
		&=\frac{1}{2N}\mathbb{E}_{\{X_k\}_{k=1}^{N}}\mathbb{E}_{\{\sigma_k\}_{k=2}^{N}}\sup_{u\in\mathcal{F}}\left[w(X_1)u(X_1)+\sum_{k=2}^{N}\sigma_kw(X_k)u(X_k)\right]\\
		&\quad\ +\frac{1}{2N}\mathbb{E}_{\{X_k\}_{k=1}^{N}}\mathbb{E}_{\{\sigma_k\}_{k=2}^{N}}\sup_{u\in\mathcal{F}}\left[-w(X_1)u(X_1)+\sum_{k=2}^{N}\sigma_kw(X_k)u(X_k)\right]\\
		&=\frac{1}{2N}\mathbb{E}_{\{X_k\}_{k=1}^{N}}\mathbb{E}_{\{\sigma_k\}_{k=2}^{N}}\\
		&\quad\sup_{u,u'\in\mathcal{F}}\left[w(X_1)[u(X_1)-u'(X_1)]+\sum_{k=2}^{N}\sigma_kw(X_k)u(X_k)+\sum_{k=2}^{N}\sigma_kw(X_k)u'(X_k)\right]\\
		&\leq\frac{1}{2N}\mathbb{E}_{\{X_k\}_{k=1}^{N}}\mathbb{E}_{\{\sigma_k\}_{k=2}^{N}}\\
		&\quad\sup_{u,u'\in\mathcal{F}}\left[\mathcal{B}|u(X_1)-u'(X_1)|+\sum_{k=2}^{N}\sigma_kw(X_k)u(X_k)+\sum_{k=2}^{N}\sigma_kw(X_k)u'(X_k)\right]\\
		&=\frac{1}{2N}\mathbb{E}_{\{X_k\}_{k=1}^{N}}\mathbb{E}_{\{\sigma_k\}_{k=2}^{N}}\\
		&\quad\sup_{u,u'\in\mathcal{F}}\left[\mathcal{B}[u(X_1)-u'(X_1)]+\sum_{k=2}^{N}\sigma_kw(X_k)u(X_k)+\sum_{k=2}^{N}\sigma_kw(X_k)u'(X_k)\right]\\
		&=\frac{1}{N}\mathbb{E}_{\{X_k,\sigma_k\}_{k=1}^{N}}\sup_{u\in\mathcal{F}}\left[\sigma_1\mathcal{B}u(X_1)+\sum_{k=2}^{N}\sigma_kw(X_k)u(X_k)\right]\\
		&\leq\cdots\leq\frac{\mathcal{B}}{N}\mathbb{E}_{\{X_k,\sigma_k\}_{k=1}^{N}}\sup_{u\in\mathcal{F}}\sum_{k=1}^{N}\sigma_ku(X_k)=\mathcal{B}\mathfrak{R}(\mathcal{F}).
	\end{align*}
\end{proof}

Now we bound the difference between continuous loss and empirical loss in terms of Rademacher complexity.
\begin{Lemma} \label{symmetrization}
\begin{align*}
	\mathbb{E}_{\{{X_i}\}_{i=1}^{N}}\sup_{u\in\mathcal{P}}\pm \left[\mathcal{L}_1(u)-\widehat{\mathcal{L}}_1(u)\right]
	&\leq C(\Omega,coe)\mathfrak{R}_N(\mathcal{F}_1),\\
	\mathbb{E}_{\{{X_i}\}_{i=1}^{N}}\sup_{u\in\mathcal{P}}\pm \left[\mathcal{L}_2(u)-\widehat{\mathcal{L}}_2(u)\right]
	&\leq C(\Omega,coe)\mathfrak{R}_N(\mathcal{F}_2),\\
	\mathbb{E}_{\{{X_i}\}_{i=1}^{N}}\sup_{u\in\mathcal{P}}\pm \left[\mathcal{L}_3(u)-\widehat{\mathcal{L}}_3(u)\right]
	&\leq C(\Omega,coe)\mathfrak{R}_N(\mathcal{F}_3),\\
	\mathbb{E}_{\{{Y_j}\}_{j=1}^{M}}\sup_{u\in\mathcal{P}}\pm \left[\mathcal{L}_4(u)-\widehat{\mathcal{L}}_4(u)\right]
	&\leq\frac{\alpha}{\beta}C(\Omega,coe)\mathfrak{R}_N(\mathcal{F}_4),\\
	\mathbb{E}_{\{{Y_j}\}_{j=1}^{M}}\sup_{u\in\mathcal{P}}\pm \left[\mathcal{L}_5(u)-\widehat{\mathcal{L}}_5(u)\right]
	&\leq\frac{1}{\beta}C(\Omega,coe)\mathfrak{R}_N(\mathcal{F}_5),
\end{align*}
where
\begin{align*}
	\mathcal{F}_1&=\{|\nabla u|^2:u\in\mathcal{P}\},
	&\mathcal{F}_2&=\{u^2:u\in\mathcal{P}\},\\
	\mathcal{F}_3&=\{u:u\in\mathcal{P}\},
	&\mathcal{F}_4&=\{u^2|_{\partial\Omega}:u\in\mathcal{P}\},\\
	\mathcal{F}_5&=\{u|_{\partial\Omega}:u\in\mathcal{P}\}.
\end{align*}
\end{Lemma}
\begin{proof}
We only present the proof with respect to $\mathcal{L}_2$ since other inequalities can be shown similarly. We take $\{\widetilde{X_k}\}_{k=1}^{N}$ as an independent copy of $\{{X_k}\}_{k=1}^{N}$, then
\begin{align*}
	\mathcal{L}_2(u)-\widehat{\mathcal{L}}_2(u)
	&=\frac{|\Omega|}{2}\left[\mathbb{E}_{X\sim U(\Omega)}w(X)u^2(X)-\frac{1}{N}\sum_{k=1}^{N}w(X_k)u^2(X_k)\right]\\	
	&=\frac{|\Omega|}{2N}\mathbb{E}_{\{\widetilde{X_k}\}_{k=1}^{N}}\sum_{k=1}^{N}\left[w(\widetilde{X_k})u^2(\widetilde{X_k})-w(X_k)u^2(X_k)\right].
\end{align*}
Hence
\begin{align*}
	&\mathbb{E}_{\{{X_k}\}_{k=1}^{N}}\sup_{u\in\mathcal{P}}\left|\mathcal{L}_2(u)-\widehat{\mathcal{L}}_2(u)\right|\\
	&\leq\frac{|\Omega|}{2N}\mathbb{E}_{\{{X_k},\widetilde{X_k}\}_{k=1}^{N}}\sup_{u\in\mathcal{P}}\sum_{k=1}^{N}\left[w(\widetilde{X_k})u^2(\widetilde{X_k})-w({X_k})u^2({X_k})\right]\\
	&=\frac{|\Omega|}{2N}\mathbb{E}_{\{{X_k},\widetilde{X_k},\sigma_k\}_{k=1}^{N}}\sup_{u\in\mathcal{P}}\sum_{k=1}^{N}\sigma_k\left[w(\widetilde{X_k})u^2(\widetilde{X_k})-w({X_k})u^2({X_k})\right]\\
	&\leq\frac{|\Omega|}{2N}\mathbb{E}_{\{\widetilde{X_k},{\sigma_k}\}_{k=1}^{N}}\sup_{u\in\mathcal{P}}\sum_{k=1}^{N}\sigma_kw(\widetilde{X_k})u^2(\widetilde{X_k})+\frac{|\Omega|}{2N}\mathbb{E}_{\{{X_k},{\sigma_k}\}_{k=1}^{N}}\sup_{u\in\mathcal{P}}\sum_{k=1}^{N}-\sigma_kw({X_k})u^2({X_k})\\
	&=\frac{|\Omega|}{N}\mathbb{E}_{\{{X_k},{\sigma_k}\}_{k=1}^{N}}\sup_{u\in\mathcal{P}}\sum_{k=1}^{N}\sigma_kw({X_k})u^2({X_k})\\
	&= |\Omega|\mathfrak{R}_N(w\cdot\mathcal{F}_2)
	\le C(\Omega,coe)\mathfrak{R}_N(\mathcal{F}_2),
\end{align*}
where the second step is due to the fact that the insertion of Rademacher variables doesn't change the distribution, the fourth step is because $\sigma_kw(\widetilde{X_k})u^2(\widetilde{X_k})$ and $-\sigma_kw({X_k})u^2({X_k})$ have the same distribution, and we use Lemma \ref{structural result of Rademacher} in the last step.
\end{proof}

In order to bound Rademecher complexities, we need the concept of covering number.
\begin{Definition}
An $\epsilon$-cover of a set $T$ in a metric space $(S, \tau)$
is a subset $T_c\subset S$ such  that for each $t\in T$, there exists a $t_c\in T_c$ such that $\tau(t, t_c) \leq\epsilon$. The $\epsilon$-covering number of $T$, denoted as $\mathcal{C}(\epsilon, T,\tau)$ is  defined to be the minimum cardinality among all $\epsilon$-cover of $T$ with respect to the metric $\tau$.
\end{Definition}

In Euclidean space, we can establish an upper bound of covering number for a bounded set easily.
\begin{Lemma} \label{covering number Euclidean space}
	Suppose that $T\subset\mathbb{R}^d$ and $\|t\|_2\leq B$ for $t\in T$, then
	\begin{equation*}
		\mathcal{C}(\epsilon,T,\|\cdot\|_2)\leq\left(\frac{2B\sqrt{d}}{\epsilon}\right)^d.
	\end{equation*}
\end{Lemma}
\begin{proof}
	Let $m=\left\lfloor\frac{2B\sqrt{d}}{\epsilon}\right\rfloor$ and define
	\begin{equation*}
		T_c=\left\{-B+\frac{\epsilon}{\sqrt{d}},-B+\frac{2\epsilon}{\sqrt{d}},\cdots,-B+\frac{m\epsilon}{\sqrt{d}}\right\}^d,
	\end{equation*}
	then for $t\in T$, there exists $t_c\in T_c$ such that
	\begin{equation*}
		\|t-t_c\|_2\leq\sqrt{\sum_{i=1}^{d}\left(\frac{\epsilon}{\sqrt{d}}\right)^2}=\epsilon.
	\end{equation*}
	Hence
	\begin{equation*}
		\mathcal{C}(\epsilon,T,\|\cdot\|_2)\leq|T_c|=m^d\leq\left(\frac{2B\sqrt{d}}{\epsilon}\right)^d.
	\end{equation*}
\end{proof}

A Lipschitz parameterization allows us to translates a cover of the function space into a cover of the parameter space. Such a property plays an essential role in our analysis of statistical error.
\begin{Lemma} \label{covering number Lipshcitz parameterization}
	Let $\mathcal{F}$ be a parameterized class of functions: $\mathcal{F} = \{f(x; \theta) : \theta\in\Theta \}$. Let $\|\cdot\|_{\Theta}$ be a norm on $\Theta$ and let $\|\cdot\|_{\mathcal{F}}$ be a norm on $\mathcal{F}$. Suppose that the mapping $\theta \mapsto f(x;\theta)$ is L-Lipschitz, that is,
	\begin{equation*}
		\left\|f(x;\theta)-f\left(x;\widetilde{\theta}\right)\right\|_{\mathcal{F}} \leq L\left\|\theta-\widetilde{\theta}\right\|_{\Theta},
	\end{equation*}
	then for any $\epsilon>0$, $\mathcal{C}\left(\epsilon, \mathcal{F},\|\cdot\|_{\mathcal{F}}\right) \leq \mathcal{C}\left(\epsilon / L, \Theta,\|\cdot\|_{\Theta}\right)$.
\end{Lemma}
\begin{proof}
	Suppose that $\mathcal{C}\left(\epsilon / L, \Theta,\|\cdot\|_{\Theta}\right)=n$ and $\{\theta_i\}_{i=1}^{n}$ is an $\epsilon / L$-cover of $\Theta$. Then for any $\theta\in\Theta$, there exists $1\leq i\leq n$ such that
	\begin{equation*}
		\left\|f(x;\theta)-f\left(x;{\theta}_i\right)\right\|_{\mathcal{F}} \leq L\left\|\theta-{\theta}_i\right\|_{\Theta}\leq\epsilon.
	\end{equation*}
	Hence $\{f(x;\theta_i)\}_{i=1}^{n}$ is an $\epsilon$-cover of $\mathcal{F}$, implying that $\mathcal{C}\left(\epsilon, \mathcal{F},\|\cdot\|_{\mathcal{F}}\right) \leq n$.
\end{proof}

To find the relation between Rademacher complexity and covering number, we first need the Massart’s finite class lemma stated below.
\begin{Lemma}\label{Mfinite}
	For any finite set $A \subset \mathbb{R}^N$ with diameter $D = \sup_{a\in A}\|a\|_2$,
	\begin{equation*}
	\mathfrak{R}_N(A) \leq \frac{D}{N}\sqrt{2\log |A|}.
	\end{equation*}
\end{Lemma}
\begin{proof}
See, for example, \cite[Lemma 26.8]{shalev2014understanding}.
\end{proof}

\begin{Lemma} \label{chaining}
Let $\mathcal{F}$ be a function class and $\|f\|_{\infty}\leq B$ for any $f\in\mathcal{F}$, we have
\begin{equation*}
\mathfrak{R}_N(\mathcal{F})\leq\inf_{0<\delta<B/2}\left(4\delta+\frac{12}{\sqrt{N}}\int_{\delta}^{B/2}\sqrt{\log\mathcal{C}(\epsilon,\mathcal{F},\|\cdot\|_{\infty})}d\epsilon\right).
\end{equation*}
\end{Lemma}
\begin{proof}
We apply chaning method. Set $\epsilon_k=2^{-k+1}B$. We denote by $\mathcal{F}_k$ such that $\mathcal{F}_k$ is an $\epsilon_k$-cover of $\mathcal{F}$ and $\left|\mathcal{F}_k\right|=\mathcal{C}(\epsilon_k,\mathcal{F},\|\cdot\|_{\infty})$. Hence for any $u\in\mathcal{F}$, there exists $u_k\in\mathcal{F}_k$ such that $\|u-u_k\|_\infty\leq\epsilon_k$. Let $K$ be a positive integer determined later. We have
\begin{align*}
	&\mathfrak{R}_N(\mathcal{F}) = \mathbb{E}_{\{\sigma_i,X_i\}_{i=1}^N} \left[\sup _{u \in \mathcal{F}} \frac{1}{N}\sum_{i=1}^{N} \sigma_{i} u\left(X_{i}\right) \right]\\\
	&=\mathbb{E}_{\{\sigma_i,X_i\}_{i=1}^N} \left[\frac{1}{N} \sup _{u \in \mathcal{F}} \sum_{i=1}^{N} \sigma_{i}\left(u\left(X_{i}\right)-u_K\left(X_{i}\right)\right)+\sum_{j=1}^{K-1} \sum_{i=1}^{N} \sigma_{i}\left(u_{j+1}\left(X_{i}\right)-u_j\left(X_{i}\right)\right)+\sum_{i=1}^{N} \sigma_{i} u_1\left(X_{i}\right)\right] \\
	&\leq \mathbb{E}_{\{\sigma_i,X_i\}_{i=1}^N} \left[\sup _{u \in \mathcal{F}} \frac{1}{N}\sum_{i=1}^{N} \sigma_{i}\left(u\left(X_{i}\right)-u_K\left(X_{i}\right)\right)\right]+\sum_{j=1}^{K-1} \mathbb{E}_{\{\sigma_i,X_i\}_{i=1}^N} \left[\sup _{u \in \mathcal{F}} \frac{1}{N}\sum_{i=1}^{N} \sigma_{i}\left(u_{j+1}\left(X_{i}\right)-u_j\left(X_{i}\right)\right)\right] \\
	&\quad+\mathbb{E}_{\{\sigma_i,X_i\}_{i=1}^N}\left[ \frac{1}{N}\sup _{u \in \mathcal{F}_1} \frac{1}{N}\sum_{i=1}^{N} \sigma_{i} u(X_i)\right].
\end{align*}
We can choose $\mathcal{F}_1=\{0\}$ to eliminate the third term. For the first term,
\begin{equation*}
	\mathbb{E}_{\{\sigma_i,X_i\}_{i=1}^N}\sup _{u \in \mathcal{F}} \frac{1}{N}\left[\sum_{i=1}^{N} \sigma_{i}\left(u\left(X_{i}\right)-u_K\left(X_{i}\right)\right)\right]
	\leq\mathbb{E}_{\{\sigma_i,X_i\}_{i=1}^N}\sup _{u \in \mathcal{F}} \frac{1}{N}\sum_{i=1}^{N} \left|\sigma_{i}\right|\left\|u-u_K\right\|_{\infty}\leq\epsilon_K.
\end{equation*}
For the second term, for any fixed samples $\{X_i\}_{i=1}^N$, we define
$$
V_j := \{(u_{j+1}\left(X_{1}\right)-u_j\left(X_{1}\right),\dots,u_{j+1}\left(X_{N}\right)-u_j\left(X_{N}\right)) \in \mathbb{R}^N: u\in\mathcal{F} \}.
$$
Then, for any $v^j\in V_j$,
\begin{align*}
	\|v^j\|_2&=\left(\sum_{i=1}^{n}\left|u_{j+1}(X_i)-u_j(X_i)\right|^2\right)^{1/2}
	\leq \sqrt{n}\left\|u_{j+1}-u_j\right\|_{\infty}\\
	&\leq\sqrt{n}\left\|u_{j+1}-u\right\|_{\infty}+\sqrt{n}\left\|u_{j}-u\right\|_{\infty}
	=\sqrt{n}\epsilon_{j+1}+\sqrt{n}\epsilon_{j}=3\sqrt{n}\epsilon_{j+1}.
\end{align*}
Applying Lemma $\ref{Mfinite}$, we have
\begin{align*}
	&\sum_{j=1}^{K-1} \mathbb{E}_{\{\sigma\}_{i=1}^N}\left[\sup _{u \in \mathcal{F}} \frac{1}{N}\sum_{i=1}^{N} \sigma_{i}\left(u_{j+1}\left(X_{i}\right)-u_j\left(X_{i}\right)\right)\right] \\
	&=\sum_{j=1}^{K-1} \mathbb{E}_{\{\sigma_i\}_{i=1}^N}\left[\sup _{v^j \in V_j} \frac{1}{N}\sum_{i=1}^{N} \sigma_{i}v_i^j\right]
	\leq\sum_{j=1}^{K-1}\frac{3\epsilon_{j+1}}{\sqrt{N}}\sqrt{2\log|V_j|}.
\end{align*}
By the denition of $V_j$, we know that $\left|V_j\right|\leq\left|\mathcal{F}_j\right|\left|\mathcal{F}_{j+1}\right|\leq\left|\mathcal{F}_{j+1}\right|^2$. Hence
\begin{align*}
	\sum_{j=1}^{K-1} \mathbb{E}_{\{\sigma_i,X_i\}_{i=1}^N} \left[\sup _{u \in \mathcal{F}} \frac{1}{N}\sum_{i=1}^{N} \sigma_{i}\left(u_{j+1}\left(X_{i}\right)-u_j\left(X_{i}\right)\right)\right] \leq\sum_{j=1}^{K-1}\frac{6\epsilon_{j+1}}{\sqrt{N}}\sqrt{\log \left|\mathcal{F}_{j+1}\right|}.
\end{align*}
Now we obtain
\begin{align*}
	\mathfrak{R}_N(\mathcal{F}) &
	\leq \epsilon_K+\sum_{j=1}^{K-1}\frac{6\epsilon_{j+1}}{\sqrt{N}}\sqrt{\log \left|\mathcal{F}_{j+1}\right|}\\
	&=\epsilon_K+\frac{12}{\sqrt{N}}\sum_{j=1}^{K-1}(\epsilon_{j+1}-\epsilon_{j+2})\sqrt{\log\mathcal{C}(\epsilon_{j+1},\mathcal{F},\|\cdot\|_{\infty})}\\
	&\leq\epsilon_{K}+\frac{12}{\sqrt{N}}\int_{\epsilon_{K+1}}^{B/2}\sqrt{\log\mathcal{C}(\epsilon,\mathcal{F},\|\cdot\|_{\infty})}d\epsilon.
\end{align*}
We conclude the lemma by choosing $K$ such that $\epsilon_{K+2}<\delta\leq\epsilon_{K+1}$ for any $0<\delta<B/2$.
\end{proof}

From Lemma $\ref{covering number Lipshcitz parameterization}$ we know that the kep step to bound $\mathcal{C}(\epsilon,\mathcal{F}_i,\|\cdot\|_{\infty})$ with $\mathcal{F}_i$ defined in Lemma $\ref{symmetrization}$ is to compute the upper bound of Lipschitz constant of class $\mathcal{F}_i$, which is done in Lemma $\ref{f Lip}$-$\ref{Lip of Fi}$.
\begin{Lemma} \label{f Lip}
Let $\mathcal{D},\mathfrak{n}_{\mathcal{D}},n_i\in\mathbb{N}^+$, $n_\mathcal{D}=1$, $B_{\theta}\ge 1$ and $\rho$ be a bounded Lipschitz continuous function with $B_{\rho}, L_{\rho}\leq 1$. Set the parameterized function class
$
\mathcal{P}=\mathcal{N}_{\rho}\left(\mathcal{D}, \mathfrak{n}_{\mathcal{D}}, B_{\theta}\right)
$. For any $f(x;\theta)\in\mathcal{P}$, $f(x;\theta)$ is $\sqrt{\mathfrak{n}_{\mathcal{D}}}B_{\theta}^{\mathcal{D}-1}\left(\prod_{i=1}^{\mathcal{D}-1}n_i\right)$-Lipschitz continuous with respect to variable $\theta$, i.e.,
\begin{equation*}
\left|f(x;\theta)-{f}(x;\widetilde{\theta})\right|
\leq\sqrt{\mathfrak{n}_{\mathcal{D}}}B_{\theta}^{\mathcal{D}-1}\left(\prod_{i=1}^{\mathcal{D}-1}n_i\right)\left\|\theta-\widetilde{\theta}\right\|_2,\quad\forall x\in\Omega.
\end{equation*}
\end{Lemma}
\begin{proof}
For $\ell=2,\cdots,\mathcal{D}$(the argument for the case of $\ell=\mathcal{D}$ is slightly different),
\begin{align*}
	\left|f_q^{(\ell)}-\widetilde{f}_q^{(\ell)}\right|
	&=\left|\rho\left(\sum_{j=1}^{n_{\ell-1}}a_{qj}^{(\ell)}f_j^{(\ell-1)}+b_q^{(\ell)}\right)-\rho\left(\sum_{j=1}^{n_{\ell-1}}\widetilde{a}_{qj}^{(\ell)}\widetilde{f}_j^{(\ell-1)}+\widetilde{b}_q^{(\ell)}\right)\right|\\
	&\leq L_{\rho}\left|\sum_{j=1}^{n_{\ell-1}}a_{qj}^{(\ell)}f_j^{(\ell-1)}-\sum_{j=1}^{n_{\ell-1}}\widetilde{a}_{qj}^{(\ell)}\widetilde{f}_j^{(\ell-1)}+b_q^{(\ell)}-\widetilde{b}_q^{(\ell)}\right|\\
	&\leq L_{\rho}\sum_{j=1}^{n_{\ell-1}}\left|a_{qj}^{(\ell)}\right|\left|f_j^{(\ell-1)}-\widetilde{f}_j^{(\ell-1)}\right|+L_{\rho}\sum_{j=1}^{n_{\ell-1}}\left|a_{qj}^{(\ell)}-\widetilde{a}_{qj}^{(\ell)}\right|\left|\widetilde{f}_j^{(\ell-1)}\right|+L_{\rho}\left|b_q^{(\ell)}-\widetilde{b}_q^{(\ell)}\right|\\
	&\leq B_{\theta}L_{\rho}\sum_{j=1}^{n_{\ell-1}}\left|f_j^{(\ell-1)}-\widetilde{f}_j^{(\ell-1)}\right|+B_{\rho}L_{\rho}\sum_{j=1}^{n_{\ell-1}}\left|a_{qj}^{(\ell)}-\widetilde{a}_{qj}^{(\ell)}\right|+L_{\rho}\left|b_q^{(\ell)}-\widetilde{b}_q^{(\ell)}\right|\\
	&\leq B_{\theta}\sum_{j=1}^{n_{\ell-1}}\left|f_j^{(\ell-1)}-\widetilde{f}_j^{(\ell-1)}\right|+\sum_{j=1}^{n_{\ell-1}}\left|a_{qj}^{(\ell)}-\widetilde{a}_{qj}^{(\ell)}\right|+\left|b_q^{(\ell)}-\widetilde{b}_q^{(\ell)}\right|.
\end{align*}
For $\ell=1$,
\begin{align*}
	\left|f_q^{(1)}-\widetilde{f}_q^{(1)}\right|
	&= \left|\rho\left(\sum_{j=1}^{n_0}a_{qj}^{(1)}x_j^+b_q^{(1)}\right)-\rho\left(\sum_{j=1}^{n_0}\widetilde{a}_{qj}^{(1)}x_j+\widetilde{b}_q^{(1)}\right)\right|\\
	&\le \sum_{j=1}^{n_{0}}\left|a_{qj}^{(1)}-\widetilde{a}_{qj}^{(1)}\right|+\left|b_q^{(1)}-\widetilde{b}_q^{(1)}\right|
	=\sum_{j=1}^{\mathfrak{n}_1}\left|\theta_j-\widetilde{\theta}_j\right|.
\end{align*}
For $\ell=2$,
\begin{align*}
	\left|f_q^{(2)}-\widetilde{f}_q^{(2)}\right|
	&\leq B_{\theta}\sum_{j=1}^{n_{1}}\left|f_j^{(1)}-\widetilde{f}_j^{(1)}\right|+\sum_{j=1}^{n_{1}}\left|a_{qj}^{(2)}-\widetilde{a}_{qj}^{(2)}\right|+\left|b_q^{(2)}-\widetilde{b}_q^{(2)}\right|\\
	&\leq B_{\theta}\sum_{j=1}^{n_{1}}\sum_{k=1}^{\mathfrak{n}_1}\left|\theta_k-\widetilde{\theta}_k\right|+\sum_{j=1}^{n_{1}}\left|a_{qj}^{(2)}-\widetilde{a}_{qj}^{(2)}\right|+\left|b_q^{(2)}-\widetilde{b}_q^{(2)}\right|\\
	&\leq n_1B_{\theta}\sum_{j=1}^{\mathfrak{n}_2}\left|\theta_j-\widetilde{\theta}_j\right|.
\end{align*}
Assuming that for $\ell\geq2$,
\begin{align*}
\left|f_q^{(\ell)}-\widetilde{f}_q^{(\ell)}\right|
\leq \left(\prod_{i=1}^{\ell-1}n_i\right) B_{\theta}^{\ell-1}\sum_{j=1}^{\mathfrak{n}_{\ell}}\left|\theta_j-\widetilde{\theta}_j\right|,
\end{align*}
we have
\begin{align*}
	\left|f_q^{(\ell+1)}-\widetilde{f}_q^{(\ell+1)}\right|
	&\leq B_{\theta}\sum_{j=1}^{n_{\ell}}\left|f_j^{(\ell)}-\widetilde{f}_j^{(\ell)}\right|+\sum_{j=1}^{n_{\ell}}\left|a_{qj}^{(\ell+1)}-\widetilde{a}_{qj}^{(\ell+1)}\right|+\left|b_q^{(\ell+1)}-\widetilde{b}_q^{(\ell+1)}\right|\\
	&\leq B_{\theta}\sum_{j=1}^{n_{\ell}}\left(\prod_{i=1}^{\ell-1}n_i\right) B_{\theta}^{\ell-1}\sum_{k=1}^{\mathfrak{n}_1}\left|\theta_k-\widetilde{\theta}_k\right|+\sum_{j=1}^{n_{\ell}}\left|a_{qj}^{(\ell+1)}-\widetilde{a}_{qj}^{(\ell+1)}\right|+\left|b_q^{(\ell+1)}-\widetilde{b}_q^{(\ell+1)}\right|\\
	&\leq \left(\prod_{i=1}^{\ell}n_i\right) B_{\theta}^{\ell}\sum_{j=1}^{\mathfrak{n}_{\ell+1}}\left|\theta_j-\widetilde{\theta}_j\right|.
\end{align*}
Hence by induction and H$\mathrm{\ddot{o}}$lder inequality we conclude that
\begin{equation*}
\left|f-\widetilde{f}\right|
\leq \left(\prod_{i=1}^{\mathcal{D}-1}n_i\right) B_{\theta}^{\mathcal{D}-1}\sum_{j=1}^{\mathfrak{n}_{\mathcal{D}}}\left|\theta_j-\widetilde{\theta}_j\right|
\leq\sqrt{\mathfrak{n}_{\mathcal{D}}}B_{\theta}^{\mathcal{D}-1}\left(\prod_{i=1}^{\mathcal{D}-1}n_i\right)\left\|\theta-\widetilde{\theta}\right\|_2.
\end{equation*}
\end{proof}

\begin{Lemma}
Let $\mathcal{D},\mathfrak{n}_{\mathcal{D}},n_i\in\mathbb{N}^+$, $n_\mathcal{D}=1$, $B_{\theta}\ge 1$ and $\rho$ be a function such that $\rho'$ is bounded by $B_{\rho'}$. Set the parameterized function class
$
\mathcal{P}=\mathcal{N}_{\rho}\left(\mathcal{D}, \mathfrak{n}_{\mathcal{D}}, B_{\theta}\right)
$. Let $p=1,\cdots,d$. We have
\begin{align*}
	\left|\partial_{x_p}f_q^{(\ell)}\right|&\leq\left(\prod_{i=1}^{\ell-1}n_i\right)\left(B_{\theta}B_{\rho'}\right)^{\ell},
	\quad\ell=1,2,\cdots,\mathcal{D}-1,\\
	\left|\partial_{x_p}f\right|&\leq\left(\prod_{i=1}^{\mathcal{D}-1}n_i\right)B_{\theta}^{\mathcal{D}}B_{\rho'}^{\mathcal{D}-1}.
\end{align*}
\end{Lemma}
\begin{proof}
For $\ell=1,2,\cdots,\mathcal{D}-1$,
\begin{align*}
	\left|\partial_{x_p}f_q^{(\ell)}\right|
	&=\left|\sum_{j=1}^{n_{\ell-1}}a_{qj}^{(\ell)}\partial_{x_p}f_j^{(\ell-1)}\rho'\left(\sum_{j=1}^{n_{\ell-1}}a_{qj}^{(\ell)}f_j^{(\ell-1)}+b_q^{(\ell)}\right)\right|
	\leq B_{\theta}B_{\rho'}\sum_{j=1}^{n_{\ell-1}}\left|\partial_{x_p}f_j^{(\ell-1)}\right|\\
	&\leq\left(B_{\theta}B_{\rho'}\right)^2\sum_{k=1}^{n_{\ell-1}}\sum_{j=1}^{n_{\ell-2}}\left|\partial_{x_p}f_j^{(\ell-2)}\right|
	=n_{\ell-1}\left(B_{\theta}B_{\rho'}\right)^2\sum_{j=1}^{n_{\ell-2}}\left|\partial_{x_p}f_j^{(\ell-2)}\right|\\
	&\leq\cdots\leq\left(\prod_{i=2}^{\ell-1}n_i\right)\left(B_{\theta}B_{\rho'}\right)^{\ell-1}\sum_{j=1}^{n_{1}}\left|\partial_{x_p}f_j^{(1)}\right| \\
	&\leq\left(\prod_{i=2}^{\ell-1}n_i\right)\left(B_{\theta}B_{\rho'}\right)^{\ell-1}\sum_{j=1}^{n_{1}}B_{\theta}B_{\rho'}=\left(\prod_{i=1}^{\ell-1}n_i\right)\left(B_{\theta}B_{\rho'}\right)^{\ell}.
\end{align*}
The bound for $\left|\partial_{x_p}f\right|$ can be derived similarly.
\end{proof}

\begin{Lemma} \label{Df Lip}
Let $\mathcal{D},\mathfrak{n}_{\mathcal{D}},n_i\in\mathbb{N}^+$, $n_\mathcal{D}=1$, $B_{\theta}\ge 1$ and $\rho$ be a function such that $\rho,\rho'$ are bounded by $B_{\rho}, B_{\rho'}\le 1$ and have Lipschitz constants $L_{\rho}, L_{\rho'}\leq 1$, respectively. Set the parameterized function class
$
\mathcal{P}=\mathcal{N}_{\rho}\left(\mathcal{D}, \mathfrak{n}_{\mathcal{D}}, B_{\theta}\right)
$. Then, for any $f(x;\theta)\in\mathcal{P}$, $p=1,\cdots,d$,  $\partial_{x_p}f(x;\theta)$ is $\sqrt{\mathfrak{n}_{\mathcal{D}}}(\mathcal{D}+1)B_{\theta}^{2\mathcal{D}}\left(\prod_{i=1}^{\mathcal{D}-1}n_i\right)^2$-Lipschitz continuous with respect to variable $\theta$, i.e.,
\begin{equation*}
	\left|\partial_{x_p}f(x;\theta)-\partial_{x_p}f(x;\widetilde{\theta})\right|
	\leq \sqrt{\mathfrak{n}_{\mathcal{D}}}(\mathcal{D}+1)B_{\theta}^{2\mathcal{D}}\left(\prod_{i=1}^{\mathcal{D}-1}n_i\right)^2\left\|\theta-\widetilde{\theta}\right\|_2,\quad \forall x\in\Omega.
\end{equation*}
\end{Lemma}
\begin{proof}
For $\ell=1$,
\begin{align*}
	&\left|\partial_{x_p}f_q^{(1)}-\partial_{x_p}\widetilde{f}_q^{(1)}\right| \\
	=&\left|a_{qp}^{(1)}\rho'\left(\sum_{j=1}^{n_{0}}a_{qj}^{(1)}x_j+b_q^{(1)}\right)-\widetilde{a}_{qp}^{(1)}\rho'\left(\sum_{j=1}^{n_{0}}\widetilde{a}_{qj}^{(1)}x_j+\widetilde{b}_q^{(1)}\right)\right|\\
	\leq&\left|a_{qp}^{(1)}-\widetilde{a}_{qp}^{(1)}\right|\left|\rho'\left(\sum_{j=1}^{n_{0}}a_{qj}^{(1)}x_j+b_q^{(1)}\right)\right|
	+\left|\widetilde{a}_{qp}^{(1)}\right|\left|\rho'\left(\sum_{j=1}^{n_{0}}a_{qj}^{(1)}x_j+b_q^{(1)}\right)-\rho'\left(\sum_{j=1}^{n_{0}}\widetilde{a}_{qj}^{(1)}x_j+\widetilde{b}_q^{(1)}\right)\right|\\
	\leq& B_{\rho'}\left|a_{qp}^{(1)}-\widetilde{a}_{qp}^{(1)}\right|+B_{\theta}L_{\rho'}\sum_{j=1}^{n_{0}}\left|a_{qj}^{(1)}-\widetilde{a}_{qj}^{(1)}\right|+B_{\theta}L_{\rho'}\left|{b}_q^{(1)}-\widetilde{b}_q^{(1)}\right|\leq 2B_{\theta}\sum_{k=1}^{\mathfrak{n}_{1}}\left|\theta_k-\widetilde{\theta}_k\right|
\end{align*}
For $\ell\geq 2$, we establish the Recurrence relation:
\begin{align*}
	&\left|\partial_{x_p}f_q^{(\ell)}-\partial_{x_p}\widetilde{f}_q^{(\ell)}\right|\\
	&\leq\sum_{j=1}^{n_{\ell-1}}\left|a_{qj}^{(\ell)}\right|\left|\partial_{x_p}f_j^{(\ell-1)}\right|\left|\rho'\left(\sum_{j=1}^{n_{\ell-1}}a_{qj}^{(\ell)}f_j^{(\ell-1)}+b_q^{(\ell)}\right)-\rho'\left(\sum_{j=1}^{n_{\ell-1}}\widetilde{a}_{qj}^{(\ell)}\widetilde{f}_j^{(\ell-1)}+\widetilde{b}_q^{(\ell)}\right)\right|\\
	&\quad+\sum_{j=1}^{n_{\ell-1}}\left|a_{qj}^{(\ell)}\partial_{x_p}f_j^{(\ell-1)}-\widetilde{a}_{qj}^{(\ell)}\partial_{x_p}\widetilde{f}_j^{(\ell-1)}\right|\left|\rho'\left(\sum_{j=1}^{n_{\ell-1}}\widetilde{a}_{qj}^{(\ell)}\widetilde{f}_j^{(\ell-1)}+\widetilde{b}_q^{(\ell)}\right)\right|\\
	&\leq B_{\theta}L_{\rho'}\sum_{j=1}^{n_{\ell-1}}\left|\partial_{x_p}f_j^{(\ell-1)}\right|\left(\sum_{j=1}^{n_{\ell-1}}\left|a_{qj}^{(\ell)}f_j^{(\ell-1)}-\widetilde{a}_{qj}^{(\ell)}\widetilde{f}_j^{(\ell-1)}\right|+\left|b_q^{(\ell)}-\widetilde{b}_q^{(\ell)}\right|\right)\\
	&\quad+B_{\rho'}\sum_{j=1}^{n_{\ell-1}}\left|a_{qj}^{(\ell)}\partial_{x_p}f_j^{(\ell-1)}-\widetilde{a}_{qj}^{(\ell)}\partial_{x_p}\widetilde{f}_j^{(\ell-1)}\right|\\
	&\leq B_{\theta}L_{\rho'}\sum_{j=1}^{n_{\ell-1}}\left|\partial_{x_p}f_j^{(\ell-1)}\right|\left(B_{\rho}\sum_{j=1}^{n_{\ell-1}}\left|a_{qj}^{(\ell)}-\widetilde{a}_{qj}^{(\ell)}\right|+B_{\theta}\sum_{j=1}^{n_{\ell-1}}\left|f_j^{(\ell-1)}-\widetilde{f}_j^{(\ell-1)}\right|+\left|b_q^{(\ell)}-\widetilde{b}_q^{(\ell)}\right|\right)\\
	&\quad+B_{\rho'}B_{\theta}\sum_{j=1}^{n_{\ell-1}}\left|\partial_{x_p}f_j^{(\ell-1)}-\partial_{x_p}\widetilde{f}_j^{(\ell-1)}\right|+B_{\rho'}\sum_{j=1}^{n_{\ell-1}}\left|a_{qj}^{(\ell)}-\widetilde{a}_{qj}^{(\ell)}\right|\left|\partial_{x_p}\widetilde{f}_j^{(\ell-1)}\right|\\
	&\leq B_{\theta}\sum_{j=1}^{n_{\ell-1}}\left|\partial_{x_p}f_j^{(\ell-1)}\right|\left(\sum_{j=1}^{n_{\ell-1}}\left|a_{qj}^{(\ell)}-\widetilde{a}_{qj}^{(\ell)}\right|+B_{\theta}\sum_{j=1}^{n_{\ell-1}}\left|f_j^{(\ell-1)}-\widetilde{f}_j^{(\ell-1)}\right|+\left|b_q^{(\ell)}-\widetilde{b}_q^{(\ell)}\right|\right)\\
	&\quad+B_{\theta}\sum_{j=1}^{n_{\ell-1}}\left|\partial_{x_p}f_j^{(\ell-1)}-\partial_{x_p}\widetilde{f}_j^{(\ell-1)}\right|+\sum_{j=1}^{n_{\ell-1}}\left|a_{qj}^{(\ell)}-\widetilde{a}_{qj}^{(\ell)}\right|\left|\partial_{x_p}\widetilde{f}_j^{(\ell-1)}\right|\\
	&\leq B_{\theta}\left(\prod_{i=1}^{\ell-1}n_i\right)B_{\theta}^{\ell}\left(\sum_{j=1}^{n_{\ell-1}}\left|a_{qj}^{(\ell)}-\widetilde{a}_{qj}^{(\ell)}\right|+B_{\theta}\sum_{j=1}^{n_{\ell-1}}
	\left(\prod_{i=1}^{\ell-2}n_i\right) B_{\theta}^{\ell-2}\sum_{k=1}^{\mathfrak{n}_{\ell-1}}\left|\theta_k-\widetilde{\theta}_k\right|+\left|b_q^{(\ell)}-\widetilde{b}_q^{(\ell)}\right|\right)\\
	&\quad+B_{\theta}\sum_{j=1}^{n_{\ell-1}}\left|\partial_{x_p}f_j^{(\ell-1)}-\partial_{x_p}\widetilde{f}_j^{(\ell-1)}\right|+\sum_{j=1}^{n_{\ell-1}}\left|a_{qj}^{(\ell)}-\widetilde{a}_{qj}^{(\ell)}\right|\left(\prod_{i=1}^{\ell-2}n_i\right)B_{\theta}^{\ell-1}\\
	&\leq B_{\theta}\sum_{j=1}^{n_{\ell-1}}\left|\partial_{x_p}f_j^{(\ell-1)}-\partial_{x_p}\widetilde{f}_j^{(\ell-1)}\right|
	+B_{\theta}^{2\ell}\left(\prod_{i=1}^{\ell-1}n_i\right)^2\sum_{k=1}^{\mathfrak{n}_{\ell}}\left|\theta_k-\widetilde{\theta}_k\right|
\end{align*}
For $\ell=2$,
\begin{align*}
	\left|\partial_{x_p}f_q^{(2)}-\partial_{x_p}\widetilde{f}_q^{(2)}\right|&\leq B_{\theta}\sum_{j=1}^{n_{1}}\left|\partial_{x_p}f_j^{(1)}-\partial_{x_p}\widetilde{f}_j^{(1)}\right|
	+B_{\theta}^{4}n_1^2\sum_{k=1}^{\mathfrak{n}_{2}}\left|\theta_k-\widetilde{\theta}_k\right|\\
	&\leq2B_{\theta}^2n_1\sum_{k=1}^{\mathfrak{n}_{1}}\left|\theta_k-\widetilde{\theta}_k\right|+B_{\theta}^{4}n_1^2\sum_{k=1}^{\mathfrak{n}_{2}}\left|\theta_k-\widetilde{\theta}_k\right|
	\leq 3B_{\theta}^{4}n_1^2\sum_{k=1}^{\mathfrak{n}_{2}}\left|\theta_k-\widetilde{\theta}_k\right|
\end{align*}
Assuming that for $\ell\geq2$,
\begin{equation*}
	\left|\partial_{x_p}f_q^{(\ell)}-\partial_{x_p}\widetilde{f}_q^{(\ell)}\right|\leq (\ell+1) B_{\theta}^{2\ell}\left(\prod_{i=1}^{\ell-1}n_i\right)^2\sum_{k=1}^{\mathfrak{n}_{\ell}}\left|\theta_k-\widetilde{\theta}_k\right|
\end{equation*}
we have
\begin{align*}
	&\left|\partial_{x_p}f_q^{(\ell+1)}-\partial_{x_p}\widetilde{f}_q^{(\ell+1)}\right| \\
	\leq& B_{\theta}\sum_{j=1}^{n_{\ell}}\left|\partial_{x_p}f_j^{(\ell)}-\partial_{x_p}\widetilde{f}_j^{(\ell)}\right|
	+B_{\theta}^{2\ell+2}\left(\prod_{i=1}^{\ell}n_i\right)^2\sum_{k=1}^{\mathfrak{n}_{\ell+1}}\left|\theta_k-\widetilde{\theta}_k\right|\\
	\leq& B_{\theta}\sum_{j=1}^{n_{\ell}}(\ell+1) B_{\theta}^{2\ell}\left(\prod_{i=1}^{\ell-1}n_i\right)^2\sum_{k=1}^{\mathfrak{n}_{\ell}}\left|\theta_k-\widetilde{\theta}_k\right|
	+B_{\theta}^{2\ell+2}\left(\prod_{i=1}^{\ell}n_i\right)^2\sum_{k=1}^{\mathfrak{n}_{\ell+1}}\left|\theta_k-\widetilde{\theta}_k\right|\\
	\leq& (\ell+2)B_{\theta}^{2\ell+2}\left(\prod_{i=1}^{\ell}n_i\right)^2\sum_{k=1}^{\mathfrak{n}_{\ell+1}}\left|\theta_k-\widetilde{\theta}_k\right|
\end{align*}
Hence by by induction and H$\mathrm{\ddot{o}}$lder inequality we conclude that
\begin{equation*}
\left|\partial_{x_p}f-\partial_{x_p}\widetilde{f}\right|
\leq (\mathcal{D}+1)B_{\theta}^{2\mathcal{D}}\left(\prod_{i=1}^{\mathcal{D}-1}n_i\right)^2\sum_{k=1}^{\mathfrak{n}_{\mathcal{D}}}\left|\theta_k-\widetilde{\theta}_k\right|
\leq \sqrt{\mathfrak{n}_{\mathcal{D}}}(\mathcal{D}+1)B_{\theta}^{2\mathcal{D}}\left(\prod_{i=1}^{\mathcal{D}-1}n_i\right)^2\left\|\theta-\widetilde{\theta}\right\|_2
\end{equation*}
\end{proof}

\begin{Lemma} \label{Lip of Fi}
Let $\mathcal{D},\mathfrak{n}_{\mathcal{D}},n_i\in\mathbb{N}^+$, $n_\mathcal{D}=1$, $B_{\theta}\ge 1$ and $\rho$ be a function such that $\rho,\rho'$ are bounded by $B_{\rho}, B_{\rho'}\le 1$ and have Lipschitz constants $L_{\rho}, L_{\rho'}\leq 1$, respectively. Set the parameterized function class
$
\mathcal{P}=\mathcal{N}_{\rho}\left(\mathcal{D}, \mathfrak{n}_{\mathcal{D}}, B_{\theta}\right)
$
Then, for any $f_i(x;\theta),f_i(x;\widetilde{\theta})\in\mathcal{F}_i$, $i=1,\cdots,5$, we have
\begin{align*}
|f_i(x;\theta)| &\le B_i, \quad \forall x\in\Omega,\\
|f_i(x;\theta)-f_i(x;\widetilde{\theta})| &\leq L_i\|\theta-\widetilde{\theta}\|_2,\quad \forall x\in\Omega,
\end{align*}
with $B_1=d\left(\prod_{i=1}^{\mathcal{D}-1}n_i\right)^2B_{\theta}^{2\mathcal{D}}$, $B_2=B_4=B_3^2$, $B_3=B_5=(n_{\mathcal{D}-1}+1)B_{\theta}$, and
\begin{align*}
L_1&=2d\sqrt{\mathfrak{n}_{\mathcal{D}}}(\mathcal{D}+1)B_{\theta}^{3\mathcal{D}}\left(\prod_{i=1}^{\mathcal{D}-1}n_i\right)^3,
&L_2&=2\sqrt{\mathfrak{n}_{\mathcal{D}}}B_{\theta}^{\mathcal{D}}(n_{\mathcal{D}-1}+1)\left(\prod_{i=1}^{\mathcal{D}-1}n_i\right),\\
L_3&=\sqrt{\mathfrak{n}_{\mathcal{D}}}B_{\theta}^{\mathcal{D}-1}\left(\prod_{i=1}^{\mathcal{D}-1}n_i\right),
&L_4&=L_2,\quad L_5=L_3.
\end{align*}
\end{Lemma}
\begin{proof}
Direct result from Lemma $\ref{f Lip}$, $\ref{Df Lip}$ and some calculation.
\end{proof}

Now we state our main result with respect to statistical error.
\begin{Theorem} \label{sta error}
Let $\mathcal{D},\mathfrak{n}_{\mathcal{D}},n_i\in\mathbb{N}^+$, $n_\mathcal{D}=1$, $B_{\theta}\ge 1$ and $\rho$ be a function such that $\rho,\rho'$ are bounded by $B_{\rho}, B_{\rho'}\le 1$ and have Lipschitz constants $L_{\rho}, L_{\rho'}\leq 1$, respectively. Set the parameterized function class
$
	\mathcal{P}=\mathcal{N}_{\rho}\left(\mathcal{D}, \mathfrak{n}_{\mathcal{D}}, B_{\theta}\right)
$. Then, if $N=M$, we have
\begin{equation*}
\mathbb{E}_{\{{X_i}\}_{i=1}^{N},\{{Y_j}\}_{j=1}^{M}}\sup_{u\in\mathcal{P}}\pm\left[\mathcal{L}(u)-\widehat{\mathcal{L}}(u)\right]\leq \frac{C(\Omega,coe,\alpha)}{\beta} \frac{d\sqrt{\mathcal{D}}\mathfrak{n}_{\mathcal{D}}^{2\mathcal{D}}B_{\theta}^{2\mathcal{D}}}{\sqrt{N}} \sqrt{\log\left(d\mathcal{D}\mathfrak{n}_{\mathcal{D}}B_{\theta}N\right)}.
\end{equation*}
\end{Theorem}
\begin{proof}
From Lemma $\ref{covering number Euclidean space}$, $\ref{covering number Lipshcitz parameterization}$ and $\ref{chaining}$, we have
\begin{align*}
	\mathfrak{R}_N(\mathcal{F}_i)&\leq\inf_{0<\delta<B_i/2}\left(4\delta+\frac{12}{\sqrt{N}}\int_{\delta}^{B_i/2}\sqrt{\log\mathcal{C}(\epsilon,\mathcal{F}_i,\|\cdot\|_{\infty})}d\epsilon\right)\\
	&\leq\inf_{0<\delta<B_i/2}\left(4\delta+\frac{12}{\sqrt{N}}\int_{\delta}^{B_i/2}\sqrt{\mathfrak{n}_{\mathcal{D}}\log\left(\frac{2L_iB_\theta\sqrt{\mathfrak{n}_{\mathcal{D}}}}{\epsilon}\right)}d\epsilon\right) \\
	&\leq\inf_{0<\delta<B_i/2}\left(4\delta+\frac{6\sqrt{\mathfrak{n}_{\mathcal{D}}}B_i}{\sqrt{N}}\sqrt{\log\left(\frac{2L_iB_\theta\sqrt{\mathfrak{n}_{\mathcal{D}}}}{\delta}\right)}\right).
\end{align*}
Choosing $\delta=1/\sqrt{N}<B_i/2$ and applying Lemma $\ref{Lip of Fi}$, we have
\begin{align}
	\mathfrak{R}_N(\mathcal{F}_i)&\leq\frac{4}{\sqrt{N}}+\frac{6\sqrt{\mathfrak{n}_{\mathcal{D}}}B_i}{\sqrt{N}}\sqrt{\log\left(2L_iB_\theta\sqrt{\mathfrak{n}_{\mathcal{D}}}\sqrt{N}\right)}\nonumber\\
	&\leq C\frac{d\sqrt{\mathfrak{n}_{\mathcal{D}}}\left(\prod_{i=1}^{\mathcal{D}-1}n_i\right)^2B_{\theta}^{2\mathcal{D}}}{\sqrt{N}}\sqrt{\log\left(4d\mathfrak{n}_{\mathcal{D}}(\mathcal{D}+1)B_{\theta}^{3\mathcal{D}+1}\left(\prod_{i=1}^{\mathcal{D}-1}n_i\right)^3\sqrt{N}\right)} \nonumber\\
	&\leq C\frac{d\sqrt{\mathcal{D}}\mathfrak{n}_{\mathcal{D}}^{2\mathcal{D}}B_{\theta}^{2\mathcal{D}}}{\sqrt{N}} \sqrt{\log\left(d\mathcal{D}\mathfrak{n}_{\mathcal{D}}B_{\theta}N\right)}\label{sta error1}
\end{align}
Combining Lemma $\ref{triangle inequality}$, $\ref{symmetrization}$ and $(\ref{sta error1})$, we obtain, if $N=M$,
\begin{equation*}
\mathbb{E}_{\{{X_i}\}_{i=1}^{N},\{{Y_j}\}_{j=1}^{M}}\sup_{u\in\mathcal{P}}\pm\left[\mathcal{L}(u)-\widehat{\mathcal{L}}(u)\right]
\leq \frac{C(\Omega,coe,\alpha)}{\beta} \frac{d\sqrt{\mathcal{D}}\mathfrak{n}_{\mathcal{D}}^{2\mathcal{D}}B_{\theta}^{2\mathcal{D}}}{\sqrt{N}} \sqrt{\log\left(d\mathcal{D}\mathfrak{n}_{\mathcal{D}}B_{\theta}N\right)}.
\end{equation*}

\end{proof}

\section{Covergence Rate for the Ritz Method}

\begin{Theorem}
Let (A) holds. Assume that $\mathcal{E}_{opt}=0$. Let $\rho$ be logistic function $\frac{1}{1+e^{-x}}$ or tanh function $\frac{e^x-e^{-x}}{e^x+e^{-x}}$. Let $u_{\phi_{\mathcal{A}}}$ be the solution of problem $(\ref{optimization})$ generated by a random solver and $\widehat{u}_{\phi}$ be an optimal solution of problem $(\ref{optimization})$.

(1)Let $u_R$ be the weak solution of Robin problem $(\ref{second order elliptic equation})(\ref{robin})$. For any $\epsilon>0$ and $\mu \in (0,1)$, set
the parameterized function class
\begin{equation*}
\mathcal{P}=\mathcal{N}_{\rho}\left(C\log(d+1),C(d,\beta)\epsilon^{-d /(1-\mu )},C(d,\beta) \epsilon^{-(9d+8)/(2-2\mu)}\right)
\end{equation*}
and number of samples
\begin{equation*}
N=M=C(d,\Omega, coe,\alpha,\beta)\epsilon^{-Cd\log(d+1)/(1-\mu)},
\end{equation*}
if the optimization error $\mathcal{E}_{opt} = \widehat{\mathcal{L}}\left(u_{\phi_{\mathcal{A}}}\right)-\widehat{\mathcal{L}}\left(\widehat{u}_{\phi}\right)\le \epsilon$, then
\begin{equation*}
\mathbb{E}_{\{{X_i}\}_{i=1}^{N},\{{Y_j}\}_{j=1}^{M}}\|u_{\phi_\mathcal{A}}-u_R\|_{H^1(\Omega)}\leq C(\Omega, coe,\alpha)\epsilon.
\end{equation*}

(2)Let $u_D$ be the weak solution of Dirichlet problem $(\ref{second order elliptic equation})(\ref{dirichlet})$. Set $\alpha=1,g=0$. For any $\epsilon>0$, let $\beta=C(coe)\epsilon$ as the penalty parameter, set the parameterized function class
\begin{equation*}
\mathcal{P}=\mathcal{N}_{\rho}\left(C\log(d+1),C(d)\epsilon^{-5d /2(1-\mu )},C(d) \epsilon^{-(45d+40)/(4-4\mu)}\right)
\end{equation*}
and number of samples
\begin{equation*}
N=M=C(d,\Omega, coe)\epsilon^{-Cd\log(d+1)/(1-\mu)},
\end{equation*}
if the optimization error $\mathcal{E}_{opt} \le \epsilon$, then
\begin{equation*}
\mathbb{E}_{\{{X_i}\}_{i=1}^{N},\{{Y_j}\}_{j=1}^{M}}\|u_{\phi_\mathcal{A}}-u_D\|_{H^1(\Omega)}\leq C(\Omega, coe)\epsilon.
\end{equation*}
\end{Theorem}

\begin{Remark}
Dirichlet boundary condition  corresponding to a constrained minimization problem, which may cause some difficulties in computation. The penalty method has been applied in finite element methods and finite volume method \cite{Babuska1973The,Maury2009Numerical}. It is also been used in deep PDEs solvers \cite{Weinan2017The,raissi2019physics,Xu2020finite} since it is not easy to construct a network with given values on the boundary. Recently, \cite{Mller2021ErrorEF,muller2020deep} also study the convergence of DRM with Dirichlet boundary condition via penalty method.
However, the  analysis in  \cite{Mller2021ErrorEF,muller2020deep} is based on some additional conditions, and we do not need these conditions to obtain the error inducing by the penalty. More importantly,   we provide the convergence rate analysis involving the statistical error caused by finite samples used in the SGD training,  while in \cite{Mller2021ErrorEF,muller2020deep} they  do  not consider the statistical error at all.
\end{Remark}

\begin{proof}
We first normalize the solution.
\begin{align*}
\inf_{\bar{u}\in\mathcal{P}}\|\bar{u}-u_R\|_{H^1(\Omega)}
&=\|u_R\|_{H^1(\Omega)}\inf_{\bar{u}\in\mathcal{P}}\left\|\frac{\bar{u}}{\|u_R\|_{H^1(\Omega)}}-\frac{u_R}{\|u_R\|_{H^1(\Omega)}}\right\|_{H^1(\Omega)}\\
&=\|u_R\|_{H^1(\Omega)}\inf_{\bar{u}\in\mathcal{P}}\left\|\bar{u}-\frac{u_R}{\|u_R\|_{H^1(\Omega)}}\right\|_{H^1(\Omega)}
\leq\frac{ C(coe)}{\beta}\inf_{\bar{u}\in\mathcal{P}}\left\|\bar{u}-\frac{u_R}{\|u_R\|_{H^1(\Omega)}}\right\|_{H^1(\Omega)}
\end{align*}
where in the third step we apply Lemma $\ref{uR regularity}$. By Lemma $\ref{uR regularity}$ and Corollary $\ref{app error}$, there exists a neural network function
\begin{equation*}
 u_{\rho}\in\mathcal{P}=\mathcal{N}_{\rho}\left(C\log (d+1),C(d)\left(\frac{1}{\beta^{3/2}\epsilon}\right)^{d /(1-\mu )},C(d) \left(\frac{1}{\beta^{3/2}\epsilon}\right)^{(9d+8)/(2-2\mu)}\right)
\end{equation*}
such that
\begin{equation*}
\left\|u_{\rho}-\frac{u_R}{\|u_R\|_{H^1(\Omega)}}\right\|_{H^1(\Omega)}\leq \beta^{3/2}\epsilon.
\end{equation*}
Hence,
\begin{align}
	\mathcal{E}_{app}&=\frac{1}{\beta}C(\Omega,coe,\alpha)\inf_{\bar{u}\in\mathcal{P}}\|\bar{u}-u_R\|_{H^1(\Omega)}^2\nonumber\\
	&\leq\frac{1}{\beta^3}C(\Omega,coe,\alpha)\inf_{\bar{u}\in\mathcal{P}}\left\|\bar{u}-\frac{u_R}{\|u_R\|_{H^1(\Omega)}}\right\|_{H^1(\Omega)}^2\leq C(\Omega,coe,\alpha)\epsilon^2\label{covergence rate1}.
\end{align}
Since $\rho,\rho'$ are bounded and Lipschitz continuous with $B_{\rho}, B_{\rho'}, L_{\rho}, L_{\rho'}\leq 1$ for $\rho=\frac{1}{1+e^{-x}}$ and $\rho=\frac{e^x-e^{-x}}{e^x+e^{-x}}$,
we can apply Theorem $\ref{sta error}$ with $\mathcal{D} = C\log(d+1)$, $\mathfrak{n}_{\mathcal{D}} =C(d)\left(\frac{1}{\beta^{3/2}\epsilon}\right)^{d /(1-\mu )}$, $B_{\theta}=C(d) \left(\frac{1}{\beta^{3/2}\epsilon}\right)^{(9d+8)/(2-2\mu)}$. Now we conclude that by setting
\begin{equation} \label{covergence rate2}
N,M=C(d,\Omega, coe,\alpha)\left(\frac{1}{\beta^{3/2}\epsilon}\right)^{Cd\log(d+1)/(1-\mu)}
\end{equation}
we have
\begin{equation} \label{covergence rate3}
	\mathcal{E}_{sta}=\mathbb{E}_{\{{X_i}\}_{i=1}^{N},\{{Y_j}\}_{j=1}^{M}} \left[\sup _{u \in \mathcal{P}} \left[\mathcal{L}(u)-\widehat{\mathcal{L}}(u)\right] + \sup _{u \in \mathcal{P}} \left[\widehat{\mathcal{L}}(u) - \mathcal{L}(u)\right]\right]\leq \epsilon^2.
\end{equation}
Combining Proposition $\ref{error decomposition}(1)$, $(\ref{covergence rate1})$ and $(\ref{covergence rate3})$ yields (1).

Setting the penalty parameter
\begin{equation} \label{covergence rate4}
\beta=C(coe)\epsilon
\end{equation}
and combining Lemma $\ref{penalty convergence}$, Proposition $\ref{error decomposition}(2)$ and $(\ref{covergence rate1})-(\ref{covergence rate4})$ yields (2).
\end{proof}

%\begin{Remark}
%Optimality
%\end{Remark}

\section{Conclusions and Extensions}\label{section:conclusion}
\par This paper provided an analysis of convergence rate for deep Ritz methods for elliptic   equations with Drichilet, Neumann and Robin boundary condition, respectively. Specifically, our study  shed light on how to set depth and width of networks and how to set the penalty parameter to achieve the desired convergence rate in terms of number of training samples.

 %and clarify how to choose the optimal penalty parameter in terms of the number of training samples.
\par There are several interesting further research directions. First, the approximation and statistical error bounds deriving here can be used for studying the nonasymptotic convergence rate for residual based method, such as PINNs. Second, the similar result may be applicable to deep Ritz methods for optimal control problems and inverse problems.

\section{Acknowledgements}
{We would like to thank Ingo G$\mathrm{\ddot{u}}$hring and  Mones Raslan for helpful discussions on approximation error.}

The work of Y. Jiao is supported in part by the National Science Foundation of China under Grant 11871474 and by the research
fund of KLATASDSMOE. The work of Y. Wang is supported in part by the Hong Kong Research Grant Council grants 16308518 and
16317416 and  HK Innovation Technology Fund ITS/044/18FX, as well as Guangdong-Hong Kong-Macao Joint Laboratory for Data-Driven Fluid Mechanics and Engineering Applications.

\bibliographystyle{siam}
\bibliography{ref}

\end{document}